\documentclass[preprint,12pt,3p]{elsarticle}

\usepackage{graphicx}
\usepackage{subfig}

\usepackage{amsmath,amssymb}

\usepackage{amsthm}
\newtheorem{theorem}{Theorem}[section]
\newtheorem{definition}[theorem]{Definition}
\newtheorem{lemma}[theorem]{Lemma}
\newtheorem{remark}[theorem]{Remark}

\journal{TBA}

\begin{document}

\begin{frontmatter}

\title{Constructing a variational quasi-reversibility method for a Cauchy problem for elliptic equations}

\author{Vo Anh Khoa\corref{cor1}\fnref{label1}}
\address[label1]{Department of Mathematics and Statistics, University of North Carolina at Charlotte, Charlotte, North Carolina 28223, USA}

\cortext[cor1]{Corresponding author.}

\ead{vakhoa.hcmus@gmail.com}

\author{Pham Truong Hoang Nhan\fnref{label2}}
\address[label2]{Department of Mathematics and Computer Science, VNUHCM-University of Science, 227 Nguyen Van Cu Str., Dist. 5, Ho Chi Minh City, Vietnam}
\ead{pthnhan1908@gmail.com }

\begin{abstract}
In the recent developments of regularization theory for inverse and ill-posed problems, a variational quasi-reversibility (QR) method has been designed to solve a class of time-reversed quasi-linear parabolic problems. Known as a PDE-based approach, this method relies on adding a suitable perturbing operator to the original problem and consequently, on gaining the corresponding fine stabilized operator, which leads us to a forward-like problem.  In this work, we establish new conditional estimates for such operators to solve a prototypical Cauchy problem for elliptic equations. This problem is based on the stationary case of the inverse heat conduction problem, where one wants to identify the heat distribution in a certain medium, given the partial boundary data. Using the new QR  method, we obtain a second-order initial value problem for a wave-type equation, whose weak solvability can be deduced using a priori estimates and compactness arguments. Weighted by a Carleman-like function, a new type of energy estimates is explored in a variational setting when we investigate the H\"older convergence rate of the proposed scheme. Besides, a linearized version of this scheme is analyzed. Numerical examples are provided to corroborate our theoretical analysis. 
\end{abstract}

\begin{keyword}
Inverse and ill-posed problems\sep quasi-reversibility method \sep convergence rates \sep energy estimates\sep Carleman weight
\MSC 65J05 \sep 65J20 \sep 35K92
\end{keyword}

\end{frontmatter}


\section{Introduction}
\label{sec1}
\subsection{Statement of the inverse problem}
Solving boundary value determination problems is one of the classical research topics in the field of inverse and ill-posed problems; cf. the survey \cite{Kabanikhin2008} for the background of some classical inverse problems for partial differential equations (PDEs). Physically, this type of problems stems from the stationary case of the inverse heat conduction problem. In this regard, we seek the unknown temperature distribution in a certain medium when informative data are given on some parts of the boundary. In terms of PDEs, this finding is governed by the so-called Cauchy problem for elliptic equations, which is highly ill-posed in the sense of Hadamard. In this work, we look for the real-valued function $u\left(x,y\right)$ in a unit rectangle $[0,1]\times[0,1]$ from the data $u_0\in H^1(0,1)$ at $x=0$, when $u$ obeys the following Laplace system:
\begin{align}\label{original1}
\begin{cases}
u_{xx}+u_{yy}=0 & \text{in }\left(0,1\right)\times\left(0,1\right),\\
u\left(x,0\right)=u\left(x,1\right)=0 & \text{for }x\in [0,1],\\
u\left(0,y\right)=u_{0}\left(y\right),u_{x}\left(0,y\right)=0 & \text{for }y\in [0,1].
\end{cases}
\end{align}
Here, the zero Neumann boundary condition at $x=0$ means that there is no heat entering or escaping at this boundary. Meanwhile, we assume the zero Dirichlet conditions at $y=0,1$. Since in real-world applications the data $u_0$ can only be measured, we suppose to have the measured data $u_0^{\varepsilon}\in H^1(0,1)$ associated with some noise level $\varepsilon\in(0,1)$, which satisfies 
\begin{align}\label{measure}
\left\Vert u_{0}^{\varepsilon}-u_{0}\right\Vert _{H^{1}\left(0,1\right)}\le\varepsilon.
\end{align}

Even though the model (\ref{original1}) is the simplest case of the Cauchy problem for elliptic equations, its extensions to more general scenarios were already mentioned in the previous works for regularization of Laplace equations. In fact, some particular generalizations have the analytical capability of reducing to the system (\ref{original1}); cf. subsection \ref{subsec:5.3} for the revisited. The existing literature on the Cauchy problems for elliptic equations is vast from theoretical and numerical viewpoints. In principle, such problems are solved for decades by distinctive approaches and thus, it is pertinent to address some fundamental researches. For instance, we would like to mention here the spectral regularization method for the sustainable development in the field of inverse and ill-posed problems; cf. e.g. \cite{Leitao2000,Qian2008,Tuan2010,Elden2009}. This method and its variants are based on stabilizing the unbounded kernels appearing in 
the explicit representation of solution. The Tikhonov-type regularization with convex and strictly convex functionals in \cite{Hao2000,Takeuchi2008,Klibanov2015,Reinhardt1999} has also received much attention in this field, throughout the minimization procedure. This method is essentially related to the variational logarithmic convexity in \cite{Falk1986,Elden2005}, where it ``convexifies" the energy functional logarithmically using a Carleman-like weight. Since this notion further concerns massive Carleman estimates, details of the convexification can be referred to the survey \cite{Klibanov2013} and references cited therein. Some other approaches should be addressed here include the moment method in \cite{Wei2003} and the integral equation approach in \cite{Sun2016,Chapko2012}.

\subsection{Goals and novelty}
This paper is aimed at enhancing our
understanding of the application of the novel quasi-reversibility (QR) method which has been designed so far in our pioneering work by \cite{Nguyen2019} for regularization of time-reversed parabolic systems. Based upon the original QR method by Latt\`es
and Lions in the textbook \cite{Lattees1967}, this is a PDE-based approach that consists in the establishment of two operators along with their conditional estimates. We call those the \emph{perturbing} and \emph{stabilized operators}. In this setting, the perturbing operator is added to the original PDE for the sake of ``absorbing" the unbounded operator (i.e. the aimed operator to be stabilized in the PDE). Then we obtain a stable approximate problem with the corresponding stabilized operator, which is usually referred to as the \emph{regularized problem}. The key idea of our new QR method lies in the fact that this addition turns the inverse problem into the forward-like problem by just acquiring the boundedness of the leading coefficient of the unbounded operator. Thereby, it reveals the perfect connection between the inverse and forward problems right in the governing equations. It is obvious that there are certainly hundreds of numerical methods to solve the forward problem and thus, our inversion method would be implemented easily.

Due to the aforementioned major feature, it is a vast potential that our QR method can be extended to many different concerns in the field of inverse problems. Since this is the inception stage of this approach, we deliberately apply it to design a regularized problem for the severely ill-posed problem (\ref{original1}) using the measurement $u_0^{\varepsilon}$ observed in (\ref{measure}). When doing so, we explore that the regularized problem for (\ref{original1}) is essentially expressed as an acoustic wave equation, which singles out one of the prominent aspects of the method we are studying. It is worth mentioning that the conditional estimates for our perturbing and stabilized operators are crucial for convergence analysis of the QR scheme, and they vary for different types of PDEs. Henceforth, our first novelty here is devoted to deriving these new conditional estimates for the Cauchy problem of elliptic equations. Especially, such estimates are established with relevance to the practical and computational aspects.

As in \cite{Nguyen2019}, to prove the convergence, we rely on a Carleman weight function to deduce a noise-scaled difference problem. In the standard variational setting, this weight plays a vital role in getting rid of large quantities involved in the difference system, which yields our second novelty (compared to many other works for inverse boundary value problems). Then our proof is centered around a weighted energy method that is coupled with careful conditional estimates for the established operators. The obtained error estimate is of the H\"older rate under some certain assumptions on the true solution. In some sense, the true solution can be assumed as a unique strong solution in the standard elliptic boundary value problem. We also analyze the convergence of a linearized version of the proposed QR scheme.

\subsection{Outline of this article}
The rest of the paper is organized as follows. Section \ref{sec:2} is devoted to the concrete setting of the QR framework we would like to study in this work. In this part, we define a regularized system for our Laplace problem (\ref{original1}). In section \ref{sec:3}, we prove the weak solvability of the regularized system using a priori estimates with compactness arguments. Convergence analysis is then conducted in section \ref{sec:4}, where we exploit a Carleman-like weight to prove the H\"older rate of convergence of the scheme. In section \ref{sec:5}, we discuss a choice of the operators we are constructing in the QR framework. Furthermore, we design a linearized version of the QR scheme and show its convergence by making a condition between the discretization in $x$ and the noise level $\varepsilon$. We also recall some generalizations of (\ref{original1}) and address some concerns about large noise levels, where our method can be modified to use. We provide two numerical examples in section \ref{sec:6} to see how our method works. Essentially, it performs very well when $\varepsilon = 10\%,1\%$; these types of noise are relevant in physical applications. We close the paper by some conclusions in \ref{sec:7} for some works in the near future.

\section{A modified quasi-reversibility framework}\label{sec:2}
In the sequel, $\left\langle \cdot,\cdot\right\rangle$ indicates either the scalar product in $L^2(0,1)$ or the dual pairing of a continuous linear functional  and an element of a function space. We mean $\left\Vert \cdot\right\Vert $ the norm in the Hilbert space $L^2(0,1)$. Different inner products and norms should be written as 
$\left\langle \cdot,\cdot\right\rangle_{X}$ and $\left\Vert \cdot\right\Vert_{X} $, respectively, where $X$ is a certain Banach space. Throughout this paper, $C$ denotes a non-negative constant which is not dependent on the noise level $\varepsilon$. This constant may vary from line to line, but we usually indicate its dependences if necessary.

From now on, we introduce the so-called regularization parameter $\beta\left(\varepsilon\right)\in (0,1)$ such that $\beta\to 0$ as $\varepsilon\to0$. In doing so, we consider an auxiliary function $\gamma:(0,1)\to \mathbb{R}$ such that for $\beta\in (0,1)$ there holds
\[
\gamma(\beta)\ge 1,\quad \lim_{\beta\to 0}\gamma(\beta)=\infty.
\]

\begin{definition}[perturbing operator]\label{def1}
	The linear mapping $\mathbf{Q}_{\varepsilon}^{\beta}:L^2(0,1)\to L^2(0,1)$ is said to be a perturbing operator if there exist a function space $\mathbb{W}\subset L^2(0,1)$ and a noise-independent constant $C_0 > 0$ such that
	\begin{align}\label{QQ}
	\left\Vert \mathbf{Q}_{\varepsilon}^{\beta}u\right\Vert \le C_{0}\left\Vert u\right\Vert _{\mathbb{W}}/\gamma(\beta)\quad\text{for any }u\in\mathbb{W}.
	\end{align}
\end{definition}

\begin{definition}[stabilized operator]\label{def2}
	The linear mapping $\mathbf{P}_{\varepsilon}^{\beta}: H^1(0,1)\to L^2(0,1)$ is said to be a stabilized operator if there exists a noise-independent constant $C_1 > 0$ such that
	\begin{align}\label{PP}
	\left\Vert \mathbf{P}_{\varepsilon}^{\beta}u\right\Vert \le C_{1}\log\left(\gamma\left(\beta\right)\right)\left\Vert u\right\Vert_{H^1(0,1)} \quad\text{for any }u\in H^{1}\left(0,1\right).
	\end{align}
\end{definition}

For each noise level, our quasi-reversibility scheme is constructed in the following manner. We perturb the PDE in \eqref{original1} by a linear mapping $\mathbf{Q}_{\varepsilon}^{\beta}$ and take $\mathbf{P}_{\varepsilon}^{\beta} = \mathbf{Q}_{\varepsilon}^{\beta} +2 \partial^{2}/\partial y^2$. Thus, we have
\begin{align}\label{IVP1}
	\frac{\partial^{2}}{\partial x^{2}}
	u_{\beta}^{\varepsilon}
	- \frac{\partial^{2}}{\partial y^{2}} u_{\beta}^{\varepsilon} + \mathbf{P}_{\varepsilon}^{\beta}u_{\beta}^{\varepsilon} = 0\quad\text{in } (0,1)\times(0,1),
\end{align}
associated with the Dirichlet boundary condition and the initial conditions:
\begin{align}\label{IVP2}
\begin{cases}
u_{\beta}^{\varepsilon}\left(x,0\right)=u_{\beta}^{\varepsilon}\left(x,1\right)=0 & \text{for }x\in[0,1],\\
u_{\beta}^{\varepsilon}\left(0,y\right)=u_{0}^{\varepsilon}\left(y\right),\partial_{x}u_{\beta}^{\varepsilon}\left(0,y\right)=0 & \text{for }y\in[0,1].
\end{cases}
\end{align}
Essentially, (\ref{IVP1})--(\ref{IVP2}) form our regularized problem.

\begin{remark}\label{rem:2.3}
	The variable $x$ in (\ref{IVP1})--(\ref{IVP2}) can be understood as a parametric time. Therefore, system (\ref{IVP1})--(\ref{IVP2}) resembles a Dirichlet--Cauchy  wave equation in a unit rectangle controlled by the noise level $\varepsilon$. Since energy of the stabilized term $\mathbf{P}_{\varepsilon}^{\beta}u_{\beta}^{\varepsilon}$ is large with respect to the noise argument, a careful adaption of fundamental energy techniques that we usually enjoy in forward problems for wave equations is really needed.
	
	According to the standard result for the Dirichlet eigenvalue problem, there exists an orthonormal basis of $L^2(0,1)$, denoted by $\left\{\phi_{j}\right\}_{j\in\mathbb{N}}$, such that $\phi_{j}\in H^1_0(0,1)\cap C^{\infty}[0,1]$ and $-\partial^2_{y^2} \phi_{j}(y) = \mu_{j}\phi_{j}(y)$ for $y\in (0,1)$. The Dirichlet eigenvalues $\left\{\mu_{j}\right\}_{j\in\mathbb{N}}$ form an infinite sequence which goes to infinity in the following sense
	\[
	0\le \mu_{0} < \mu_{1} < \mu_{2} < \ldots, \text{ and }\lim_{j\to\infty}\mu_{j} = \infty.
	\]
\end{remark}

\begin{remark}
	The conditional estimate (\ref{PP}) is weaker than the one we have proposed in \cite{Nguyen2019} because of the following three reasons. First, the regularized problem in this work is expressed as a hyperbolic equation; compared to the parabolic one in \cite{Nguyen2019}, which turns out that we eventually need the information in $H^1(0,1)$ during the energy estimates. Second, we mimic the Fourier truncation method to design a computable stabilized operator. It is due to the fact that the stabilized operators we have introduced in \cite{Nguyen2019} are formed by a Fourier series with modified kernels, and it is hard to approximate an infinite series without truncating high frequencies in a suitable manner. Finally, the space $\mathbb{W}$ for the true solution we need for our convergence analysis is usually a Gevrey-type space, which is not natural to be assumed according to the forward problem for linear elliptic equations. This is completely different from the Gevrey assumption on the true solution we have made in \cite{Nguyen2019} because it is well-known that for some analytical parabolic equations they possess themselves a local-in-time weak solution in Gevrey spaces; cf. \cite{Ferrari1998}. In the present work, we rely on a special property of solution using the Fourier representation to show a $C^1$ regularity bound for a particular Gevrey criterion we need. This bound is a perfect match for the estimate (\ref{PP}). The second and third reasons can be manifested in subsection \ref{subsec:5.1}.
\end{remark}

\section{Weak solvability of the regularized system (\ref{IVP1})--(\ref{IVP2})}\label{sec:3}
To show the weak solvability of the regularized problem, we introduce its weak formulation in the following manner.

\begin{definition}\label{weaksol}
	For each $\varepsilon>0$, a function $u_{\beta}^{\varepsilon}:[0,1]\to H_0^1(0,1)$ is said to be a weak solution to system (\ref{IVP1})--(\ref{IVP2}) if
	\begin{itemize}
		\item $u_{\beta}^{\varepsilon}\in C([0,1];H_0^1(0,1))$, $\partial_{x}u_{\beta}^{\varepsilon}\in C([0,1];L^2(0,1))$, $\partial_{x^2}^2 u_{\beta}^{\varepsilon}\in L^2(0,1;H^{-1}(0,1))$;
		
		\item For every test function $\psi\in H_0^1(0,1)$, it holds that
		\begin{align}\label{IVP3}
		\left\langle \frac{\partial^{2}}{\partial x^{2}}u_{\beta}^{\varepsilon},\psi\right\rangle
		+ \left\langle \partial_{y}u_{\beta}^{\varepsilon},\partial_{y}\psi\right\rangle
		 +\left\langle \mathbf{P}_{\varepsilon}^{\beta}u_{\beta}^{\varepsilon},\psi\right\rangle =0\quad \text{for a.e. in }(0,1);
		\end{align}
		
		\item $u_{\beta}^{\varepsilon}(0)=u_{0}^{\varepsilon}$ and $\partial_{x}u_{\beta}^{\varepsilon}(0)=0$.
	\end{itemize}
	
\end{definition}

Consider the $n$-dimensional subspace $\mathbb{S}_{n}$ of $H_0^1(0,1)$ generated by $\phi_0,\phi_1,\ldots,\phi_{n}$. Using the Galerkin projection, we construct approximate solutions to (\ref{IVP1})--(\ref{IVP2}) in the form
\begin{align}\label{Galerkin}
u_{n}^{\varepsilon}(x,y)=\sum_{j=0}^{n}U_{jn}^{\varepsilon}(x)\phi_{j}(y).
\end{align}
Note that for ease of presentation, we neglect the presence of $\beta$ in this section. The solution $u_{n}^{\varepsilon}$ is the solution of the following approximate problem:
\begin{align}\label{IVP4}
\left\langle \frac{\partial^{2}}{\partial x^{2}}u_{n}^{\varepsilon},\psi\right\rangle
+ \left\langle \partial_{y}u_{n}^{\varepsilon},\partial_{y}\psi\right\rangle
+\left\langle \mathbf{P}_{\varepsilon}^{\beta}u_{n}^{\varepsilon},\psi\right\rangle =0\quad \text{for }\psi\in\mathbb{S}_{n}\;\text{and a.e. in }(0,1),
\end{align}
with $\partial_{x}u_{n}^{\varepsilon}\left(0\right)=0$ and
\begin{align}\label{IVP5}
u_{n}^{\varepsilon}\left(0\right)=\sum_{j=0}^{n}\left(U_{0}^{\varepsilon}\right)_{jn}\phi_{j}\to u_{0}^{\varepsilon}\quad\text{strongly in }H^{1}\left(0,1\right)\;\text{as }n\to\infty.
\end{align}
By the choice $\psi = \phi_{j}$, the functions  $U_{jn}^{\varepsilon}$ in (\ref{Galerkin}) are to be found as  solutions to  the Cauchy problem for the system of $n$ ordinary differential equations:
\begin{align}
& \frac{d^{2}}{dx^{2}}U_{jn}^{\varepsilon}+\mu_{j}U_{jn}^{\varepsilon}+\sum_{i=0}^{n}U_{in}^{\varepsilon}\left\langle \mathbf{P}_{\varepsilon}^{\beta}\phi_{i},\phi_{j}\right\rangle =0, \label{aa}\\
& U_{jn}^{\varepsilon}\left(0\right)=\left(U_{0}^{\varepsilon}\right)_{jn} ,\frac{d}{dx}U_{jn}^{\varepsilon}\left(0\right)=0.\label{bb}
\end{align}

\begin{lemma}\label{lem:C1}
	For any fixed $n\in\mathbb{N}$ and for each $\varepsilon>0$, system (\ref{aa})--(\ref{bb}) has a unique solution $U_{jn}^{\varepsilon}\in C^1([0,1])$.
\end{lemma}
\begin{proof}
	Let $Z_{jn}^{\varepsilon}=\frac{d}{dx}U_{jn}^{\varepsilon}$. It follows from system (\ref{aa})--(\ref{bb}) that
	\begin{align*}
	& \frac{d}{dx}\begin{bmatrix}U_{jn}^{\varepsilon}\\
	Z_{jn}^{\varepsilon}
	\end{bmatrix}=\begin{bmatrix}0 & 1\\
	-\mu_{j} & 0
	\end{bmatrix}\begin{bmatrix}U_{jn}^{\varepsilon}\\
	Z_{jn}^{\varepsilon}
	\end{bmatrix}+\begin{bmatrix}0\\
	-\sum_{i=0}^{n}U_{in}^{\varepsilon}\left\langle \mathbf{P}_{\varepsilon}^{\beta}\phi_{i},\phi_{j}\right\rangle 
	\end{bmatrix},\\
	& \begin{bmatrix}U_{jn}^{\varepsilon}\left(0\right)\\
	Z_{jn}^{\varepsilon}\left(0\right)
	\end{bmatrix}=\begin{bmatrix}\left(U_{0}^{\varepsilon}\right)_{jn}\\
	0
	\end{bmatrix}.
	\end{align*}
	Consider $z_{jn}^{\varepsilon}=\left[U_{jn}^{\varepsilon},Z_{jn}^{\varepsilon}\right]^{\text{T}}$. We thus obtain the following integral equation:
	\begin{align}\label{ccc}
		z_{jn}^{\varepsilon}\left(x\right)=z_{jn}^{\varepsilon}\left(0\right)+A_j\int_{0}^{x}z_{jn}^{\varepsilon}\left(s\right)ds+\int_{0}^{x}F_{j}\left(z^{\varepsilon}\right)\left(s\right)ds,
	\end{align}
	where
	\[
	A_{j}=\begin{bmatrix}0 & 1\\
	-\mu_{j} & 0
	\end{bmatrix},\quad F_{j}\left(z^{\varepsilon}\right)=\begin{bmatrix}0\\
	-\sum_{i=0}^{n}U_{in}^{\varepsilon}\left\langle \mathbf{P}_{\varepsilon}^{\beta}\phi_{i},\phi_{j}\right\rangle 
	\end{bmatrix}.
	\]
	Here, we mean $z^{\varepsilon}=\left[z_{0n}^{\varepsilon},z_{1n}^{\varepsilon}\ldots,z_{nn}^{\varepsilon}\right]\in \mathbb{R}^{2(n+1)}$ and for simplicity, we, from now on, neglect the presence of $n$ in $U_{jn}^{\varepsilon}$ and $z_{jn}^{\varepsilon}$ in this proof. The integral equation (\ref{ccc}) can be rewritten as $z^{\varepsilon}(x) = H[z^{\varepsilon}](x)$, where the same notation as $z^{\varepsilon}$ is applied to $H$ with $H_{j}$ being the right-hand side of (\ref{ccc}).
	
	Define the norm in $Y=C([0,1];\mathbb{R}^{2(n+1)})$ as follows:
	\[
	\left\Vert c\right\Vert _{Y}:=\sup_{x\in\left[0,1\right]}\sum_{j=0}^{n}\left|c_{j}\left(x\right)\right| \quad\text{with }c=\left[c_{j}\right] \in \mathbb{R}^{2(n+1)}.
	\]
	We then want to prove that there exists $n_0\in \mathbb{N}^{*}$ such that the operator $H^{n_0}:=H[H^{n_0-1}]:Y\to Y$ is a contraction mapping. In other words, we find $K\in [0,1)$ such that
	\[
	\left\Vert H^{n_{0}}\left[z_1^{\varepsilon}\right]-H^{n_{0}}\left[z_2^{\varepsilon}\right]\right\Vert _{Y}\le K\left\Vert z_1^{\varepsilon}-z_2^{\varepsilon}\right\Vert _{Y}\quad\text{for any }z_1^{\varepsilon},z_2^{\varepsilon}\in Y.
	\]
	This can be done by induction. Indeed, observe that
	\begin{align*}
		& \left|H_{j}\left[z_1^{\varepsilon}\right]\left(x\right)-H_{j}\left[z_2^{\varepsilon}\right]\left(x\right)\right|\\
		& \le\int_{0}^{x}\left(\sqrt{\mu_{j}^2 + 1}\left|z_{1j}^{\varepsilon}\left(s\right)-z_{2j}^{\varepsilon}\left(s\right)\right|+C_{1}\log\left(\gamma\left(\beta\right)\right)\sum_{i=0}^{n}\left|U_{1i}^{\varepsilon}\left(s\right)-U_{2i}^{\varepsilon}\left(s\right)\right|\right)ds\\
		& \le\left(\frac{\sqrt{\mu_{j}^2+1}}{n+1}+CC_{1}\log\left(\gamma\left(\beta\right)\right)\right)x\left\Vert z_1^{\varepsilon}-z_2^{\varepsilon}\right\Vert _{Y},
	\end{align*}
	aided by the conditional estimate (\ref{PP}). Here, we indicate $C=C\left(\left\Vert \phi_{n}\right\Vert_{H_0^1(0,1)}\right)>0$. Therefore, it is immediate to prove that
	\[
	\left|H_{j}^{n}\left[z_1^{\varepsilon}\right]\left(x\right)-H_{j}^{n}\left[z_2^{\varepsilon}\right]\left(x\right)\right|\le\left[\frac{\sqrt{\mu_{j}^2+1}}{n+1}+CC_{1}\log\left(\gamma\left(\beta\right)\right)\right]^{n}\frac{x^{n}}{n!}\left\Vert z_1^{\varepsilon}-z_2^{\varepsilon}\right\Vert _{Y},
	\]
	which leads to the fact that
	\[
	\left\Vert H^{n}\left[z_1^{\varepsilon}\right]-H^{n}\left[z_2^{\varepsilon}\right]\right\Vert _{Y}\le\frac{\left\Vert z_1^{\varepsilon}-z_2^{\varepsilon}\right\Vert _{Y}}{n!}\sum_{j=0}^{n}\left[\frac{\sqrt{\mu_{j}^2+1}}{n+1}+CC_{1}\log\left(\gamma\left(\beta\right)\right)\right]^{n}.
	\]
	In view of the fact that
	\[
	\lim_{n\to\infty}\frac{1}{n!}\sum_{j=0}^{n}\left[\frac{\sqrt{\mu_{j}^2+1}}{n+1}+CC_{1}\log\left(\gamma\left(\beta\right)\right)\right]^{n}=0,
	\]
	we can find a sufficiently large $n_0$ such that
	\[
	\frac{1}{n_0!}\sum_{j=0}^{n_0}\left[\frac{\sqrt{\mu_{j}^2+1}}{n_0+1}+CC_{1}\log\left(\gamma\left(\beta\right)\right)\right]^{n_0}<1.
	\]
	This indicates the existence of $K\in [0,1)$ and that $H^{n_0}$ is a contraction mapping from $Y$ onto itself. By the Banach fixed-point argument, there exists a unique solution $z^{\varepsilon}\in Y$ such that $H^{n_0}[z^{\varepsilon}]=z^{\varepsilon}$. Since $H^{n_0}[H[z^{\varepsilon}]]=H[H^{n_0}[z^{\varepsilon}]]=H[z^{\varepsilon}]$, then the integral equation $H[z^{\varepsilon}]=z^{\varepsilon}$ admits a unique solution in $Y$. Hence, we complete the proof of the lemma.
\end{proof}

Now we can state the existence result for (\ref{IVP1})--(\ref{IVP2}) in the following theorem.

\begin{theorem}\label{thm:3.3}
	Assume (\ref{measure}) holds. For each $\varepsilon>0$, the regularized system (\ref{IVP1})--(\ref{IVP2}) admits a weak solution $u_{\beta}^{\varepsilon}$ in the sense of Definition \ref{weaksol}. Moreover, it holds that $u_{\beta}^{\varepsilon}\in C([0,1];H_0^1(0,1))$ and $\partial_{x}u_{\beta}^{\varepsilon}\in C([0,1];L^2(0,1))$.
\end{theorem}
\begin{proof}
	To prove this theorem, we need to derive some energy estimates for approximate solutions $u_n^{\varepsilon}$. Thanks to Lemma \ref{lem:C1}, we can prove that $\partial_{x}u_n^{\varepsilon}\in C([0,1];\mathbb{S}_{n})$. Taking in (\ref{IVP4}) $\psi =\partial_{x}u_n^{\varepsilon}$, we find that
	\begin{align}
	\frac{d}{dx}\left[\left\Vert \partial_{x}u_{n}^{\varepsilon}\right\Vert ^{2}+\left\Vert \partial_{y}u_{n}^{\varepsilon}\right\Vert ^{2}\right] & =-2\left\langle \mathbf{P}_{\varepsilon}^{\beta}u_{n}^{\varepsilon},\partial_{x}u_{n}^{\varepsilon}\right\rangle \nonumber \\
	& \le C_{1}\log\left(\gamma\left(\beta\right)\right)\left\Vert u_{n}^{\varepsilon}\right\Vert _{H^{1}\left(0,1\right)}^{2}+C_{1}\log\left(\gamma\left(\beta\right)\right)\left\Vert \partial_{x}u_{n}^{\varepsilon}\right\Vert ^{2}. \label{est1}
	\end{align}
	The norms $\left\Vert u_{n}^{\varepsilon}\right\Vert_{H_0^1(0,1)}=\left\Vert \partial_{y}u_{n}^{\varepsilon}\right\Vert$ and $\left\Vert u_{n}^{\varepsilon}\right\Vert_{H^1(0,1)}$ are equivalent\footnote{By the Poincar\'e inequality, we particularly have $\frac{1}{\sqrt{2}}\left\Vert u_{n}^{\varepsilon}\right\Vert_{H^1(0,1)}\le \left\Vert u_{n}^{\varepsilon}\right\Vert_{H_0^1(0,1)} \le \left\Vert u_{n}^{\varepsilon}\right\Vert_{H^1(0,1)}$.} based upon the zero trace. Thereupon, by integrating the estimate (\ref{est1}) with respect to $x$, we arrive at
	\begin{align*}
	& \left\Vert \partial_{x}u_{n}^{\varepsilon}\left(x,\cdot\right)\right\Vert ^{2}+\left\Vert u_{n}^{\varepsilon}\left(s,\cdot\right)\right\Vert _{H_{0}^{1}\left(0,1\right)}^{2}\\
	& \le\left\Vert \partial_{y}u_{n}^{\varepsilon}\left(0,\cdot\right)\right\Vert ^{2}+2C_{1}\log\left(\gamma\left(\beta\right)\right)\int_{0}^{x}\left(\left\Vert u_{n}^{\varepsilon}\left(s,\cdot\right)\right\Vert _{H_{0}^{1}\left(0,1\right)}^{2}+\left\Vert \partial_{x}u_{n}^{\varepsilon}\left(s,\cdot\right)\right\Vert ^{2}\right)ds.
	\end{align*}
	Using Gronwall's inequality, we get
	\begin{align}\label{key1}
	\left\Vert \partial_{x}u_{n}^{\varepsilon}\left(x,\cdot\right)\right\Vert ^{2}+\left\Vert u_{n}^{\varepsilon}\left(x,\cdot\right)\right\Vert _{H_{0}^{1}\left(0,1\right)}^{2}\le\left\Vert \partial_{y}u_{n}^{\varepsilon}\left(0,\cdot\right)\right\Vert ^{2}\gamma^{2C_{1}x}(\beta).
	\end{align}
	In view of (\ref{IVP5}), it is straightforward to see that $\left\Vert \partial_{y}u_{n}^{\varepsilon}\left(0,\cdot\right)\right\Vert\le C$ for any $\varepsilon > 0$. Thus, for any $n\in\mathbb{N}$ we have
	\begin{align*}
	& \gamma^{-C_{1}}\left(\beta\right)u_{n}^{\varepsilon}\;\text{is bounded in }L^{\infty}\left(0,1;H_{0}^{1}\left(0,1\right)\right),\\
	& \gamma^{-C_{1}}\left(\beta\right)\partial_{x}u_{n}^{\varepsilon}\;\text{is bounded in }L^{\infty}\left(0,1;L^{2}\left(0,1\right)\right).
	\end{align*}
	It follows from the Banach--Alaoglu theorem, and the argument that a weak limit of derivatives is the derivative of the weak limit, that we can extract a subsequence of scaled approximate solutions $\gamma^{-C_{1}}\left(\beta\right) u_{n}^{\varepsilon}$, which we still denote by $\left\{\gamma^{-C_{1}}\left(\beta\right)u_{n}^{\varepsilon}\right\}_{n\in\mathbb{N}}$, such that for each $\varepsilon>0$,
	\begin{align}
	&\label{weak1} \gamma^{-C_{1}}\left(\beta\right)u_{n}^{\varepsilon}\to\gamma^{-C_{1}}\left(\beta\right)u^{\varepsilon}\;\text{weakly}-*\;\text{in }L^{\infty}\left(0,1;H_{0}^{1}\left(0,1\right)\right),\\
	&\label{weak2} \gamma^{-C_{1}}\left(\beta\right)\partial_{x}u_{n}^{\varepsilon}\to\gamma^{-C_{1}}\left(\beta\right)\partial_{x}u^{\varepsilon}\;\text{weakly}-*\;\text{in }L^{\infty}\left(0,1;L^{2}\left(0,1\right)\right).
	\end{align}
	We remark that (\ref{weak1}) holds if and only if
	\[
	\int_{0}^{1}\left\langle \gamma^{-C_{1}}\left(\beta\right)u_{n}^{\varepsilon}\left(x,\cdot\right),w\right\rangle dx\to\int_{0}^{1}\left\langle \gamma^{-C_{1}}\left(\beta\right)u^{\varepsilon}\left(x,\cdot\right),w\right\rangle dx\quad\text{for }w\in L^{1}\left(0,1;H^{-1}\left(0,1\right)\right),
	\]
	which implies that
	\begin{align} \label{weakk4}
	u_{n}^{\varepsilon}\to u^{\varepsilon}\;\text{weakly}-*\;\text{in }L^{\infty}\left(0,1;H_{0}^{1}\left(0,1\right)\right).
	\end{align}
	The same result is obtained for the weak-star convergence in (\ref{weak2}), viz.
	\begin{align}\label{weak4}
		\partial_{x}u_{n}^{\varepsilon}\to\partial_{x}u^{\varepsilon}\;\text{weakly}-*\;\text{in }L^{\infty}\left(0,1;L^{2}\left(0,1\right)\right).
	\end{align}
	Now, using the Galerkin equation (\ref{IVP4}) we can show that $\partial_{x^2}^2u_{n}^{\varepsilon}\in L^2(0,1;\mathbb{S}_{n})$. Thus, we obtain the uniform bound with respect to $n$ of $\partial_{x^2}^2u_{n}^{\varepsilon}$ as follows:
	\begin{align*}
	& \gamma^{-C_{1}}\left(\beta\right)\left\Vert \partial_{x^{2}}^{2}u_{n}^{\varepsilon}\left(x,\cdot\right)\right\Vert _{H^{-1}\left(0,1\right)}=\sup_{\psi\in H^{1}\left(0,1\right)\backslash\left\{ 0\right\} }\frac{\gamma^{-C_{1}}\left(\beta\right)\left\langle \partial_{x^{2}}^{2}u_{n}^{\varepsilon}\left(x,\cdot\right),\psi\right\rangle }{\left\Vert \psi\right\Vert _{H_0^{1}\left(0,1\right)}}\\
	& =\sup_{\psi\in H^{1}\left(0,1\right)\backslash\left\{ 0\right\} }\frac{-\gamma^{-C_{1}}\left(\beta\right)\left\langle \partial_{y}u_{n}^{\varepsilon}\left(x,\cdot\right),\partial_{y}\psi\right\rangle -\gamma^{-C_{1}}\left(\beta\right)\left\langle \mathbf{P}_{\varepsilon}^{\beta}u_{n}^{\varepsilon}\left(x,\cdot\right),\partial_{y}\psi\right\rangle }{\left\Vert \psi\right\Vert _{H_0^{1}\left(0,1\right)}}\\
	& \le C\gamma^{-C_{1}}\left(\beta\right)\left\Vert u_{n}^{\varepsilon}\left(x,\cdot\right)\right\Vert _{H_{0}^{1}\left(0,1\right)}.
	\end{align*}
	Henceforth, by squaring this estimate, integrating the resulting with respect to $x$ and then using the Banach--Alaoglu theorem we can choose the subsequence of $\gamma^{-C_{1}}\left(\beta\right) u_{n}^{\varepsilon}$ so that
	\[
	\gamma^{-C_{1}}\left(\beta\right)\partial_{x^{2}}^{2}u_{n}^{\varepsilon}\to\gamma^{-C_{1}}\left(\beta\right)\partial_{x^{2}}^{2}u^{\varepsilon}\;\text{weakly in }L^{2}\left(0,1;H^{-1}\left(0,1\right)\right),
	\]
	which leads to
	\[
	\partial_{x^{2}}^{2}u_{n}^{\varepsilon}\to\partial_{x^{2}}^{2}u^{\varepsilon}\;\text{weakly in }L^{2}\left(0,1;H^{-1}\left(0,1\right)\right).
	\]
	Combining the above weak-star and weak limits, the function $u^{\varepsilon}$ satisfies
	\begin{align}\label{weak3}
	u^{\varepsilon}\in L^{\infty}\left(0,1;H_{0}^{1}\left(0,1\right)\right),\quad\partial_{x}u^{\varepsilon}\in L^{\infty}\left(0,1;L^{2}\left(0,1\right)\right),\quad\partial_{x^{2}}^{2}u^{\varepsilon}\in L^{2}\left(0,1;H^{-1}\left(0,1\right)\right).
	\end{align}
	Furthermore, using the Aubin--Lions lemma in combination with the Rellich--Kondrachov theorem $H_0^1(0,1)\subset L^2(0,1)$ for (\ref{weakk4}) and (\ref{weak4}), we get
	\begin{align}\label{weak6}
		u_{n}^{\varepsilon}\to u^{\varepsilon}\;\text{strongly}\;\text{in }L^{2}\left(0,1;H_0^1\left(0,1\right)\right).
	\end{align}
	
	Now we multiply both sides of the Galerkin equation (\ref{IVP4}) by an $x$-dependent test function $\tilde{w}\in C_{c}^{\infty}(0,1)$, then by integrating the resulting equation with respect to $x$ we arrive at
	\[
	\int_{0}^{1}\left\langle \frac{\partial^{2}}{\partial x^{2}}u_{n}^{\varepsilon},\upsilon\right\rangle dx+\int_{0}^{1}\left\langle \partial_{y}u_{n}^{\varepsilon},\partial_{y}\upsilon\right\rangle dx+\int_{0}^{1}\left\langle \mathbf{P}_{\varepsilon}^{\beta}u_{n}^{\varepsilon},\upsilon\right\rangle dx=0,
	\]
	where we have denoted by $\upsilon = \upsilon(x,y) = \tilde{w}(x)\psi(y)$ for $\psi\in\mathbb{S}_{n}$. Hence, passing the limit of this equation as $n\to\infty$ we obtain
	\begin{align}\label{limit}
	\int_{0}^{1}\left\langle \frac{\partial^{2}}{\partial x^{2}}u^{\varepsilon},\upsilon\right\rangle dx+\int_{0}^{1}\left\langle \partial_{y}u^{\varepsilon},\partial_{y}\upsilon\right\rangle dx+\int_{0}^{1}\left\langle \mathbf{P}_{\varepsilon}^{\beta}u^{\varepsilon},\upsilon\right\rangle dx=0,
	\end{align}
	where convergence of the second and third terms is guaranteed by  (\ref{weak6}). Equation (\ref{limit}) holds for $v=\tilde{w}\psi$ with $\psi\in H_0^1(0,1)$ and since $\tilde{w}\in C_{c}^{\infty}(0,1)$ is arbitrary, we deduce that the function $u^{\varepsilon}$ obtained from approximate solutions $u_n^{\varepsilon}$ satisfies the weak formulation (\ref{IVP3}) for every test function $\psi\in H_0^1(0,1)$.
	In addition, the arguments in (\ref{weak3}) enable us to show that
	\begin{align}\label{weak5}
	u^{\varepsilon}\in C([0,1];H_0^1(0,1)),\quad \partial_{x}u^{\varepsilon}\in C([0,1];L^2(0,1)),
	\end{align}
	where we have applied the Aubin--Lions lemma. Note that for the latter argument in (\ref{weak5}), we rely on the
	Gelfand triple $H_0^1(0,1)\subset L^2(0,1)\subset H^{-1}(0,1)$.
	
	It now remains to verify the initial data. Take $\kappa\in C^1([0,1])$ satisfying $\kappa(0)=1$ and $\kappa(1)=0$. It follows from (\ref{weak4}) that
	\[
	\int_{0}^{1}\left\langle \partial_{x}u_{n}^{\varepsilon},\psi\right\rangle \kappa\left(x\right)dx\to\int_{0}^{1}\left\langle \partial_{x}u^{\varepsilon},\psi\right\rangle \kappa\left(x\right)dx\quad\text{for }\psi\in L^{2}\left(0,1\right).
	\]
	Then by integration by parts, we have
	\[
	-\int_{0}^{1}\left\langle u_{n}^{\varepsilon},\psi\right\rangle \kappa_{x}dx-\left\langle u_{n}^{\varepsilon}\left(0\right),\psi\right\rangle \kappa\left(0\right)\to-\int_{0}^{1}\left\langle u^{\varepsilon},\psi\right\rangle \kappa_{x}dx-\left\langle u^{\varepsilon}\left(0\right),\psi\right\rangle \kappa\left(0\right)
	\]
	and thereupon, we get $\left\langle u_{n}^{\varepsilon}\left(0\right),\psi\right\rangle \to\left\langle u^{\varepsilon}\left(0\right),\psi\right\rangle $ for all $\psi\in H_0^1(0,1)$ by virtue of (\ref{weakk4}). Cf. (\ref{IVP5}) for the strong convergence of $u_n^{\varepsilon}(0)$ in $H^1(0,1)$, we obtain $\left\langle u_{n}^{\varepsilon}\left(0\right),\psi\right\rangle \to\left\langle u_{0}^{\varepsilon},\psi\right\rangle $ for all $\psi\in H^1(0,1)$. Hence, $\left\langle u^{\varepsilon}\left(0\right),\psi\right\rangle =\left\langle u_{0}^{\varepsilon},\psi\right\rangle $ for all $\psi \in H^1(0,1)$, which implies that $u^{\varepsilon}\left(0\right)=u_{0}^{\varepsilon}$ for a.e. in $(0,1)$. We complete the proof of the theorem.   
\end{proof}

Now we close this section by uniqueness of the weak solution of (\ref{IVP1})--(\ref{IVP2}). Proof of this result is still ended up with the use of some energy estimates and with exploiting the Gronwall inequality, eventually.

\begin{theorem}\label{thm:3.4}
	Assume (\ref{measure}) holds. For each $\varepsilon >0$, the regularized system (\ref{IVP1})--(\ref{IVP2}) admits a unique weak solution $u_{\beta}^{\varepsilon}$ in the sense of Definition \ref{weaksol}.
\end{theorem}
\begin{proof}
	We sketch out some important steps because this proof is standard. Indeed, let $u_1^{\varepsilon}$ and $u_2^{\varepsilon}$ be two weak solutions that we have obtained in Theorem \ref{thm:3.3}. We can see that the function $d^{\varepsilon} = u_1^{\varepsilon} - u_2^{\varepsilon}$ satisfies a linear wave equation with zero data ($d^{\varepsilon}=\partial_{x}d^{\varepsilon} = 0$) by virtue of the linearity of $\mathbf{P}_{\beta}^{\varepsilon}$. Similar to (\ref{IVP3}), the equation for $d^{\varepsilon}\in C([0,1];H_0^1(0,1))$ reads as
	\[
	\left\langle \frac{\partial^{2}}{\partial x^{2}}d^{\varepsilon},\psi\right\rangle
	+ \left\langle \partial_{y}d^{\varepsilon},\partial_{y}\psi\right\rangle
	+\left\langle \mathbf{P}_{\varepsilon}^{\beta}d^{\varepsilon},\psi\right\rangle =0.
	\]
	Then taking $\psi = \partial_{x}d^{\varepsilon}$, we proceed as
	in the way to get the estimate (\ref{key1}) in proof of Theorem \ref{thm:3.3}. Hence, $d^{\varepsilon}=0$ a.e. in $(0,1)$ because of the fact that
	\[
	\left\Vert \partial_{x}d^{\varepsilon}\left(x,\cdot\right)\right\Vert ^{2}+\left\Vert d^{\varepsilon}\left(x,\cdot\right)\right\Vert _{H_{0}^{1}\left(0,1\right)}^{2}\le 0.
	\]
	This completes the proof of the theorem.
\end{proof}

\section{Convergence analysis}\label{sec:4}

In this section, we are concerned about the rate of convergence of the regularized problem $u_{\beta}^{\varepsilon}$ towards a certain true solution $u$. The estimate is pointwise in $\varepsilon$. We come up with the error estimation by using a Carleman weight function when looking for a  scaled difference $u_{\beta}^{\varepsilon} - u$. This way is different from the proof of weak solvability of the regularized system (\ref{IVP1})--(\ref{IVP2}) because of the presence of the perturbing operator acting on the true solution and of the large conditional estimate of the stabilized operator in the difference formulation. It is worth mentioning that the scaled difference equation is essentially a damped wave equation, where the damped terms play a crucial role in controlling all involved large quantities. The original notion behind the choice of Carleman weights for convergence analysis in the field of inverse and ill-posed problems is the designation of a strictly convex cost functional, which is usually called \emph{convexification}; see e.g. the monograph \cite{Beilina2012} for a complete background of this method. In this work, we choose an exponentially decreasing weight such that it ``maximizes" the presence of initial data $u_{0}^{\varepsilon}$.

\begin{theorem}\label{thm:err1}
	Assume (\ref{measure}) holds. Let $u\in C([0,1];\mathbb{W})$ be a unique solution of the Laplace system (\ref{original1}), where the space $\mathbb{W}$ is obtained by the perturbation $\mathbf{Q}_{\varepsilon}^{\beta}$ in Definition \ref{def1}. Let $u_{\beta}^{\varepsilon}$ be a unique weak solution of the regularized system (\ref{IVP1})--(\ref{IVP2}) analyzed in Theorems \ref{thm:3.3}--\ref{thm:3.4}. By choosing $\gamma(\beta)\ge 1$ such that $2C_1\log\left(\gamma\left(\beta\right)\right)/3>1$, the following error estimate holds:
	\begin{align}\label{err1}
	& \left\Vert \partial_{x}u_{\beta}^{\varepsilon}\left(x,\cdot\right)-\partial_{x}u\left(x,\cdot\right)\right\Vert ^{2}+\left\Vert u_{\beta}^{\varepsilon}\left(x,\cdot\right)-u\left(x,\cdot\right)\right\Vert_{H^1(0,1)}^{2}\\
	& \le\frac{4}{3}C_{1}^{2}\left[4\varepsilon^{2}+\frac{27}{8}C_{0}^{2}C_{1}^{-3}\gamma^{-2}\left(\beta\right)\log^{-3}(\gamma(\beta))x\left\Vert u\right\Vert _{C\left([0,1];\mathbb{W}\right)}^{2}\right]\log^{2}\left(\gamma\left(\beta\right)\right)\gamma^{7C_{1}x/3}(\beta).\nonumber
	\end{align}
\end{theorem}
\begin{proof}
	Let $w = \left(u_{\beta}^{\varepsilon} - u\right)\exp\left(-\rho_{\beta} x\right)$ for $\rho_{\beta} > 0$  being chosen later. From (\ref{IVP1}) and (\ref{original1}), we compute the following difference equation:
	\begin{align}\label{w}
	w_{xx}-w_{yy}+2\rho_{\beta}w_{x}+\rho_{\beta}^{2}w=-\mathbf{P}_{\varepsilon}^{\beta}w+e^{-\rho_{\beta}x}\mathbf{Q}_{\varepsilon}^{\beta}u\quad\text{in } (0,1)\times(0,1).
	\end{align}
	This equation is associated with the Dirichlet boundary condition and the initial conditions:
	\begin{align}\label{ww}
	\begin{cases}
	w\left(x,0\right)=w\left(x,1\right)=0 & \text{for }x\in [0,1],\\
	w\left(0,y\right)=u_{0}^{\varepsilon}\left(y\right) - u_0(y),\partial_{x}w\left(0,y\right)= - \rho_{\beta}w(0,y) & \text{for }y\in [0,1].
	\end{cases}
	\end{align}
	Multiplying both sides of (\ref{w}) by $w_x$ and integrating the resulting equation with respect to $y$ from 0 to 1, we have
	\begin{align}
		\frac{d}{dx}\left\Vert w_{x}\left(x,\cdot\right)\right\Vert ^{2}&+\frac{d}{dx}\left\Vert w_{y}\left(x,\cdot\right)\right\Vert ^{2}+2\rho_{\beta}^{2}\frac{d}{dx}\left\Vert w\left(x,\cdot\right)\right\Vert ^{2}+4\rho_{\beta}\left\Vert w_{x}\left(x,\cdot\right)\right\Vert ^{2} \nonumber\\
		&=2e^{-\rho_{\beta}x}\left\langle \mathbf{Q}_{\varepsilon}^{\beta}u,w_{x}\right\rangle -2\left\langle \mathbf{P}_{\varepsilon}^{\beta}w,w_{x}\right\rangle =:I_{1}+I_{2}.\label{w11}
	\end{align}
	Using the conditional estimates (\ref{QQ}), (\ref{PP}) and applying the Cauchy--Schwarz inequality, we estimate $I_1$ and $I_2$ in (\ref{w11}) in the following manner:
	\begin{align*}
	I_{1} & \le C_{0}^{2}e^{-2\rho_{\beta}x}\gamma^{-2}\left(\beta\right)\rho_{\beta}^{-1}\left\Vert u\left(x,\cdot\right)\right\Vert _{\mathbb{W}}^{2}+\rho_{\beta}\left\Vert w_{x}\left(x,\cdot\right)\right\Vert ^{2},\\
	 I_{2}&\le 2C_1 \log\left(\gamma(\beta)\right)\left(\left\Vert w\left(x,\cdot\right)\right\Vert \left\Vert w_x\left(x,\cdot\right)\right\Vert + \left\Vert w_y\left(x,\cdot\right)\right\Vert \left\Vert w_x\left(x,\cdot\right)\right\Vert
	 \right)
	 \\&\le 
	 C_{1}\log\left(\gamma\left(\beta\right)\right)\left(\left\Vert w\left(x,\cdot\right)\right\Vert ^{2}+ \left\Vert w_y\left(x,\cdot\right)\right\Vert ^{2}\right)+2C_1\log\left(\gamma\left(\beta\right)\right)\left\Vert w_{x}\left(x,\cdot\right)\right\Vert ^{2}.
	\end{align*}
	Henceforth, the left-hand side of (\ref{w11}) is bounded by
	\begin{align}
	&\rho_{\beta}^{-2}\frac{d}{dx}\left\Vert w_{x}\left(x,\cdot\right)\right\Vert ^{2}+
	\rho_{\beta}^{-2}\frac{d}{dx}\left\Vert w_{y}\left(x,\cdot\right)\right\Vert ^{2}+ 2\frac{d}{dx}\left\Vert w\left(x,\cdot\right)\right\Vert ^{2} \nonumber \\
	& \le C_{0}^{2}e^{-2\rho_{\beta}x}\gamma^{-2}\left(\beta\right)\rho_{\beta}^{-3}\left\Vert u\left(x,\cdot\right)\right\Vert _{\mathbb{W}}^{2}+C_{1}\log\left(\gamma\left(\beta\right)\right)\rho_{\beta}^{-2}\left(\left\Vert w\left(x,\cdot\right)\right\Vert ^{2} + \left\Vert w_y\left(x,\cdot\right)\right\Vert ^{2}\right) \nonumber \\
	& +\rho_{\beta}^{-2}\left(2C_1\log\left(\gamma\left(\beta\right)\right)-3\rho_{\beta}\right)\left\Vert w_{x}\left(x,\cdot\right)\right\Vert ^{2}.\label{w12}
	\end{align}
	Now, choosing in (\ref{w12}) that $\rho_{\beta} = 2C_1\log\left(\gamma\left(\beta\right)\right)/3>1$,  integrating the resulting estimate with respect to $x$ from 0 to $x_1$, we get
	\begin{align}
	& \rho_{\beta}^{-2}\left(\left\Vert w_{x}\left(x_{1},\cdot\right)\right\Vert ^{2}+\left\Vert w_{y}\left(x_{1},\cdot\right)\right\Vert ^{2}\right)+\left\Vert w\left(x_{1},\cdot\right)\right\Vert ^{2} \nonumber\\
	& \le\rho_{\beta}^{-2}\left(\left\Vert w_{x}\left(0,\cdot\right)\right\Vert ^{2}+\left\Vert w_{y}\left(0,\cdot\right)\right\Vert ^{2}\right)+2\left\Vert w\left(0,\cdot\right)\right\Vert ^{2}+C_{0}^{2}\gamma^{-2}\left(\beta\right)\rho_{\beta}^{-3}\int_{0}^{x_{1}}e^{-2\rho_{\beta}x}\left\Vert u\left(x,\cdot\right)\right\Vert _{\mathbb{W}}^{2} \nonumber\\
	& +C_1\log(\gamma(\beta))\int_{0}^{x_{1}}\left(\left\Vert w\left(x,\cdot\right)\right\Vert ^{2} + \rho_{\beta}^{-2}\left\Vert w_y\left(x,\cdot\right)\right\Vert ^{2}\right)dx.
	\label{haha}
	\end{align}
	Due to (\ref{measure}), we estimate that
	\begin{align} \label{intialest}
	\rho_{\beta}^{-2}\left(\left\Vert w_{x}\left(0,\cdot\right)\right\Vert ^{2}+\left\Vert w_{y}\left(0,\cdot\right)\right\Vert ^{2}\right)+2\left\Vert w\left(0,\cdot\right)\right\Vert ^{2}=3\left\Vert u_{0}^{\varepsilon}-u_{0}\right\Vert ^{2}+\rho_{\beta}^{-2}\left\Vert \partial_{y}u_{0}^{\varepsilon}-\partial_{y}u_{0}\right\Vert ^{2}\le4\varepsilon^{2}.
	\end{align}
	Thus, using Gronwall's inequality we continue to estimate (\ref{haha}) as follows:
	\begin{align}\label{ye}
	\rho_{\beta}^{-2}\left(\left\Vert w_{x}\left(x_{1},\cdot\right)\right\Vert ^{2}+\left\Vert w_{y}\left(x_{1},\cdot\right)\right\Vert ^{2}\right)+\left\Vert w\left(x_{1},\cdot\right)\right\Vert ^{2}\le\left[4\varepsilon^{2}+C_{0}^{2}\gamma^{-2}\left(\beta\right)\rho_{\beta}^{-3}x_1\left\Vert u\right\Vert _{C\left([0,1];\mathbb{W}\right)}^{2}\right]\gamma^{C_1 x_1}(\beta).
	\end{align}
	It is now to use back-substitutions to get back the difference $u_{\beta}^{\varepsilon}-u$. In view of the fact that
	\begin{align*}
	\rho_{\beta}^{-2}e^{-2\rho_{\beta}x_{1}}\left\Vert \partial_{x}u_{\beta}^{\varepsilon}\left(x_{1},\cdot\right)-\partial_{x}u\left(x_{1},\cdot\right)\right\Vert ^{2} & \le\rho_{\beta}^{-2}\left(\rho_{\beta}\left\Vert w\left(x_{1},\cdot\right)\right\Vert +\left\Vert w_{x}\left(x_{1},\cdot\right)\right\Vert ^{2}\right)^{2}\\
	& \le2\left(\left\Vert w\left(x_{1},\cdot\right)\right\Vert ^{2}+\rho_{\beta}^{-2}\left\Vert w_{x}\left(x_{1},\cdot\right)\right\Vert ^{2}\right),
	\end{align*}
	we obtain the aimed estimate (\ref{err1}). Hence, we complete the proof of the theorem.
\end{proof}

Without loss of generality, we now assume that $C_1 < 6/7$ for our convergence analysis below. This assumption on $C_1$ is somehow similar to what has been assumed in \cite{Nguyen2019} for the parabolic case, where $C_1$ and the time domain are related to each other. In subsection \ref{subsec:5.1} below, we show a way to control the largeness of such $C_1$ in the framework of the truncation projection. Besides, the quantity $C_1$ can be largely selected due to the flexible choice of $\gamma(\beta)$.

\begin{theorem}\label{thm:err2} Under the assumptions of Theorem \ref{thm:err1}, if one has $C_1<6/7$ and chooses $\gamma=\gamma(\beta)\ge 1$ such that there holds
	\begin{align}\label{gamma}
	\lim_{\varepsilon\to0}\varepsilon\gamma\left(\beta\right)\log\left(\gamma\left(\beta\right)\right)=K_{1}\in\left(0,\infty\right),
	\end{align}
	then the following error estimate holds
	\begin{align}\label{err2}
	\left\Vert \partial_{x}u_{\beta}^{\varepsilon}\left(x,\cdot\right)-\partial_{x}u\left(x,\cdot\right)\right\Vert ^{2}+\left\Vert u_{\beta}^{\varepsilon}\left(x,\cdot\right)-u\left(x,\cdot\right)\right\Vert _{H^{1}\left(0,1\right)}^{2}\le C\gamma^{\frac{7C_{1}x}{3}-2}(\beta).
	\end{align}
	Consequently, if instead of (\ref{gamma}) we choose a stronger choice
	\begin{align}\label{gamma1}
	\lim_{\varepsilon\to0}\varepsilon\gamma\left(\beta\right)=K_{2}\in\left(0,\infty\right),
	\end{align}
	then it holds that
	\begin{align}\label{err3}
	\left\Vert u_{\beta}^{\varepsilon}\left(x,\cdot\right)-u\left(x,\cdot\right)\right\Vert ^{2}\le C\gamma^{\frac{7C_{1}x}{3}-2}(\beta).
	\end{align}
\end{theorem}
\begin{proof} As a by-product of the rigorous estimate (\ref{err1}), if the choice (\ref{gamma}) holds, one can deduce that
	\begin{align*}
	& \left\Vert \partial_{x}u_{\beta}^{\varepsilon}\left(x,\cdot\right)-\partial_{x}u\left(x,\cdot\right)\right\Vert ^{2}+\left\Vert u_{\beta}^{\varepsilon}\left(x,\cdot\right)-u\left(x,\cdot\right)\right\Vert _{H^{1}\left(0,1\right)}^{2}\\
	& \le\frac{4}{3}C_{1}^{2}\left[4\varepsilon^{2}\log^{2}\left(\gamma\left(\beta\right)\right)\gamma^{2}\left(\beta\right)+\frac{27}{8}C_{0}^{2}C_{1}^{-3}\log^{-1}\left(\gamma\left(\beta\right)\right)x\left\Vert u\right\Vert _{C\left([0,1];\mathbb{W}\right)}^{2}\right]\gamma^{\frac{7C_{1}x}{3}-2}(\beta)\\
	& \le\frac{4}{3}C_{1}^{2}\left[4K_{1}^{2}+\frac{27}{8}C_{0}^{2}C_{1}^{-3}\log^{-1}\left(\gamma\left(\beta\right)\right)x\left\Vert u\right\Vert _{C\left([0,1];\mathbb{W}\right)}^{2}\right]\gamma^{\frac{7C_{1}x}{3}-2}(\beta).
	\end{align*}
	Since $C_1 < 6/7$, we have $\gamma^{7C_1 x_1 /3 - 2}(\beta) \to 0$ as $\beta \to 0$, which guarantees the strong convergence of the scheme. Notice that if in (\ref{ye}) we drop the first two terms on the right-hand side (i.e. the gradient terms), we, after back-substitution, arrive at
	\begin{align}\label{hoho}
	\left\Vert u_{\beta}^{\varepsilon}\left(x,\cdot\right)-u\left(x,\cdot\right)\right\Vert ^{2}\le\left[4\varepsilon^{2}+C_{0}^{2}\gamma^{-2}\left(\beta\right)\rho_{\beta}^{-3}x\left\Vert u\right\Vert _{C\left([0,1];\mathbb{W}\right)}^{2}\right]\gamma^{7C_1 x/3 }.
	\end{align}
	Thus, using a stronger choice (\ref{gamma1}) (compared to (\ref{gamma})) we get
	\begin{align}\label{hohooo}
	\left\Vert u_{\beta}^{\varepsilon}\left(x,\cdot\right)-u\left(x,\cdot\right)\right\Vert ^{2}\le\left[4K_{2}^{2}+\frac{27}{8}C_{0}^{2}C_{1}^{-3}\log^{-3}\left(\gamma\left(\beta\right)\right)x\left\Vert u\right\Vert _{C\left([0,1];\mathbb{W}\right)}^{2}\right]\gamma^{\frac{7C_{1}x}{3}-2}(\beta).
	\end{align}
	Hence, we complete the proof of the theorem.
\end{proof}

\begin{remark}\label{rem:4.3} It is easy to see that if we adapt the choice (\ref{gamma}) to the estimate (\ref{hoho}), we then obtain a better rate of convergence than (\ref{err3}) in the following sense:
	\begin{align*}
	\left\Vert u_{\beta}^{\varepsilon}\left(x,\cdot\right)-u\left(x,\cdot\right)\right\Vert ^{2}\le C\gamma^{\frac{7C_{1}x}{3}-2}(\beta)\log^{-2}(\gamma(\beta)).
	\end{align*}
	As a consequence of (\ref{err2}), we rely on the compact embedding $H^1(0,1)\subset C[0,1]$ to conclude the H\"older convergence in $C([0,1]\times[0,1])$, viz.
	\[
	\sup_{\left(x,y\right)\in\left[0,1\right]^{2}}\left|u_{\beta}^{\varepsilon}\left(x,y\right)-u\left(x,y\right)\right|^{2}\le C\gamma^{\frac{7C_{1}}{3}-2}(\beta).
	\]
	As to the choice (\ref{gamma}), we can particularly take $\gamma(\beta) = \beta^{-1}$ with $\beta=\varepsilon^{1/2}$. Meanwhile, choosing $\gamma(\beta)= \beta^{-1}$ with $\beta = \varepsilon$ fulfills the choice (\ref{gamma1}). In principle, they are of the H\"older rate of convergence. Last but not least, we remark that since $2C_1\log\left(\gamma\left(\beta\right)\right)/3>1$ is acquired in the convergence analysis, we suppose a sufficiently small noise level $\varepsilon < e^{-3/C_1}$ for the particular choices have been made above.

\end{remark}

\section{Some crucial remarks and generalizations}\label{sec:5}
\subsection{A choice of the perturbing and stabilized operators in numerics}\label{subsec:5.1}
	The existence of the perturbing and stabilized operators has been pointed out in \cite{Nguyen2019}, theoretically. From the simulation standpoint, we need to truncate high Fourier frequencies during the simulation part. Nevertheless, we know that truncation of frequencies should be an appropriate noise-dependent procedure. In this regard, convergence is no longer guaranteed because the numerical perturbing and stabilized operators we use may not satisfy their established definitions. Therefore, such theoretical operators are not practical. In this work, we rely on the so-called Fourier truncation method to propose the following operator:
	\begin{align}\label{Q}
	\mathbf{Q}_{\varepsilon}^{\beta}u=2\sum_{\mu_{j}>\frac{1}{16}\log^2\left(\gamma\left(\beta\right)\right)}\mu_{j}\left\langle u,\phi_{j}\right\rangle \phi_{j}.
	\end{align}
	We can show in (\ref{Q}) that $\left\Vert \mathbf{Q}_{\varepsilon}^{\beta}u\right\Vert \le2\left\Vert u\right\Vert _{H^{2}\left(0,1\right)}$ for any $u\in H^2(0,1)$, which is obviously better than (\ref{QQ}), but our convergence analysis cannot work without the decay behavior of $\mathbf{Q}_{\varepsilon}^{\beta}$ in (\ref{QQ}). Aided by the Parseval identity, we estimate $\mathbf{Q}_{\varepsilon}^{\beta}$ in (\ref{Q}) as follows:
	\begin{align}\label{QQQ}
	\left\Vert \mathbf{Q}_{\varepsilon}^{\beta}u\right\Vert ^{2}\le\gamma^{-2}\left(\beta\right)\left\Vert \left(-\Delta e^{\sqrt{-\Delta}}\right)u\right\Vert ^{2},
	\end{align}
	where we have rewritten (\ref{QQ}) as
	\begin{align*}
	\mathbf{Q}_{\varepsilon}^{\beta}u & =2\sum_{\mu_{j}>\frac{1}{16}\log^2\left(\gamma\left(T,\beta\right)\right)}e^{-\sqrt{\mu_{j}}}\mu_{j}e^{\sqrt{\mu_{j}}}\left\langle u,\phi_{j}\right\rangle \phi_{j}.
	\end{align*}
	In (\ref{QQQ}), we obtain $\mathbb{W}$ as a characterization of a Gevrey space $G^{s/2}_{\sigma}$ with $s=2,\sigma=1$. This Gevrey space was postulated in \cite[Definition 2.1]{Tuan2015} for the first time in the field of inverse and ill-posed problems. Cf. \cite{Khoa2017}, we have the following relation for (\ref{original1}):
	\begin{align}\label{rela}
	\mu_{j}e^{\left(1-x\right)\sqrt{\mu_{j}}}\left(\left\langle u\left(x,\cdot\right),\phi_{j}\right\rangle +\frac{\left\langle u_{x}\left(x,\cdot\right),\phi_{j}\right\rangle }{\sqrt{\mu_{j}}}\right)=\mu_{j}\left\langle u\left(1,\cdot\right),\phi_{j}\right\rangle +\sqrt{\mu_{j}}\left\langle u_{x}\left(1,\cdot\right),\phi_{j}\right\rangle .
	\end{align}
	This means that it suffices to assume the true solution satisfies $u(1,\cdot)\in H^2(0,1)$ and $u_x(1,\cdot)\in H^1(0,1)$ to fulfill the Gevrey space $\mathbb{W}$. It is because of the fact that
	\begin{align}\label{rela1}
	\sup_{x\in\left[0,1\right]}\left\Vert \left(-\Delta e^{\sqrt{-\Delta}}\right)u\left(x,\cdot\right)\right\Vert ^{2}\le\sup_{x\in\left[0,1\right]}\left[\sum_{j\in\mathbb{N}}\mu_{j}^{2}e^{2\left(1-x\right)\sqrt{\mu_{j}}}\left(\left\langle u\left(x,\cdot\right),\phi_{j}\right\rangle +\frac{\left\langle u_{x}\left(x,\cdot\right),\phi_{j}\right\rangle }{\sqrt{\mu_{j}}}\right)^{2}\right].
	\end{align}
	Proofs of (\ref{rela}) and (\ref{rela1}) can be found in \ref{appendix-sec1}.
	On the order hand, this assumption is suitable when we consider the standard strong solution in $H^{2}((0,1)\times(0,1))$ of the Laplace equation, the forward system of (\ref{original1}), using the standard Sobolev embedding $H^2(0,1)\subset C^1[0,1]$.
	
	Now, using (\ref{Q}) we obtain the following stabilized operator:
	\begin{align}\label{Pnum}
	\mathbf{P}_{\varepsilon}^{\beta}u=-2\sum_{\mu_{j}\le\frac{1}{16}\log^2\left(\gamma\left(\beta\right)\right)}\mu_{j}\left\langle u,\phi_{j}\right\rangle \phi_{j}=-2\sum_{\mu_{j}\le\frac{1}{16}\log^2\left(\gamma\left(\beta\right)\right)}\mu_{j}^{1/2}\mu_{j}^{1/2}\left\langle u,\phi_{j}\right\rangle \phi_{j},
	\end{align}
	which essentially leads to (\ref{PP}) with $C_1 = 1/2$. Observe that we can control $C_1$ based upon the Fourier frequencies we want to cut off.

\subsection{A linearized version of the QR scheme}
Even though the QR scheme designed in (\ref{IVP1})--(\ref{IVP2}) is a linear mapping with respect to the regularized solution $u_{\beta}^{\varepsilon}$, it is still hard to apply it to the simulation regime. It is because of our concrete choice of operators formulated by a truncated Fourier series in subsection \ref{subsec:5.1}. It is then reasonable to propose a linearization for this QR scheme. Cf. \cite{Long2002}, we investigate convergence of the following numerical scheme for (\ref{IVP1})--(\ref{IVP2}):
\begin{align}
k\ge1: & \quad \frac{\partial^{2}}{\partial x^{2}}u_{\beta}^{\varepsilon,k}-\frac{\partial^{2}}{\partial y^{2}}u_{\beta}^{\varepsilon,k}=-\mathbf{P}_{\varepsilon}^{\beta}u_{\beta}^{\varepsilon,k-1},\nonumber \\
k=0: & \quad u_{\beta}^{\varepsilon,0}=0.\label{linear}
\end{align}
Here, for $k\ge 1$ we associate the PDE with the Dirichlet boundary condition and the initial conditions:
\[
\begin{cases}
u_{\beta}^{\varepsilon,k}\left(x,0\right)=u_{\beta}^{\varepsilon,k}\left(x,1\right)=0 & \text{for }x\in[0,1],\\
u_{\beta}^{\varepsilon,k}\left(0,y\right)=u_{0}^{\varepsilon}\left(y\right),\partial_{x}u_{\beta}^{\varepsilon,k}\left(0,y\right)=0 & \text{for }y\in[0,1].
\end{cases}
\]
The weak solvability of this linearization scheme can be proceeded as in section \ref{sec:3}. Therefore, we skip it in this part. However, we note that the right-hand side of the PDE for $k=1$ vanishes since $\mathbf{P}_{\varepsilon}^{\beta}u_{\beta}^{\varepsilon,0}=0$. Thereby, the weak solvability for $u_{\beta}^{\varepsilon,1}$ is attained for any $\varepsilon>0$. In other words, for any $\varepsilon$ one has
\begin{align}\label{u1}
\left\Vert u_{\beta}^{\varepsilon,1}\right\Vert _{C\left(\left[0,1\right];H^{1}\left(0,1\right)\right)}+\left\Vert \partial_{x}u_{\beta}^{\varepsilon,1}\right\Vert _{C\left(\left[0,1\right];L^{2}\left(0,1\right)\right)}\le C.
\end{align}

Our convergence analysis below shows that the refinement in $x$ should be dependent of the noise level $\varepsilon$. This means that the linearization is a local-in-$x$ approximation of the regularized solution $u^{\varepsilon}_{\beta}$. Nevertheless, since the convergence of $u^{\varepsilon}_{\beta}$ to the true solution $u$ is global in $x$ (cf. section \ref{sec:4}), the local approximation under consideration does not affect the whole  convergence of the QR scheme. It is because of the fact that for every $\varepsilon>0$ we can divide the domain of $x$ into many finite sub-domains and repeat the linearization procedure in every sub-domain. Hence, it suffices to consider the linearization scheme (\ref{linear}) in a sub-domain $[0,\bar{x}]\subset[0,1]$ and below, we single out the choice of $\bar{x}:=\bar{x}_{\beta}>0$ based upon $\gamma(\beta)$ to ensure the strong convergence of our linearization in $\mathbb{H}_{\bar{x}}$, where
\[
\mathbb{H}_{\bar{x}}:=\left\{u\in C\left(\left[0,\bar{x}\right];H^{1}\left(0,1\right)\right):u_{x}\in C\left(\left[0,\bar{x}\right];L^{2}\left(0,1\right)\right)\right\} .
\].

\begin{theorem}\label{thm:err3}
	The approximate solution  $u_{\beta}^{\varepsilon,k}$ defined in (\ref{linear}) is strongly convergent in $\mathbb{H}_{\bar{x}}$. Furthermore, for each $\varepsilon>0$ we can find a sufficiently small $\eta_{\beta}\in (0,1)$ such that
	\[
	\left\Vert u_{\beta}^{\varepsilon,k}-u_{\beta}^{\varepsilon}\right\Vert _{\mathbb{H}_{\bar{x}}}\le\frac{\eta_{\beta}^{k}}{1-\eta_{\beta}}\left\Vert u_{\beta}^{\varepsilon,1}\right\Vert _{\mathbb{H}_{\bar{x}}}.
	\]
\end{theorem}
\begin{proof}
	Similar to our proofs above, we herein rely on energy estimates for the difference $v^{k+1}=u_{\beta}^{\varepsilon,k+1}-u_{\beta}^{\varepsilon,k}$. Different from the difference equation (\ref{w}), this time we do not have the presence of the perturbing operator $\mathbf{Q}_{\varepsilon}^{\beta}$. The difference equation for $v^{k+1}$ reads as
	\[
	\frac{\partial^{2}}{\partial x^{2}}v^{k+1}-\frac{\partial^{2}}{\partial y^{2}}v^{k+1}=-\mathbf{P}_{\varepsilon}^{\beta}v^{k},
	\]
	along with the zero Dirichlet boundary condition and zero initial data. Multiplying the difference equation by $v_x^{k+1}$ and integrating the resulting equation from 0 to $1$, we have
	\begin{align}
	& \frac{d}{dx}\left(\left\Vert v_{x}^{k+1}\left(x,\cdot\right)\right\Vert ^{2}+\left\Vert v_{y}^{k+1}\left(x,\cdot\right)\right\Vert ^{2}\right) \nonumber\\
	& =-2\left\langle \mathbf{P}_{\varepsilon}^{\beta}v^{k},v^{k+1}\right\rangle \le2C_{1}\log\left(\gamma\left(\beta\right)\right)\left(\left\Vert v^{k}\right\Vert ^{2}+\left\Vert v_{y}^{k}\right\Vert ^{2}\right)^{1/2}\left\Vert v_{x}^{k+1}\right\Vert . \label{hehe}
	\end{align}
	Integrating (\ref{hehe}) from $0$ to $x$, we estimate that
	\begin{align*}
	& \left\Vert v_{x}^{k+1}\left(x,\cdot\right)\right\Vert ^{2}+\left\Vert v_{y}^{k+1}\left(x,\cdot\right)\right\Vert ^{2}\\
	& \le C_{1}\log\left(\gamma\left(\beta\right)\right)\int_{0}^{x}\left(\left\Vert v^{k}\left(s,\cdot\right)\right\Vert ^{2}+\left\Vert v_{y}^{k}\left(s,\cdot\right)\right\Vert ^{2}+\left\Vert v_{x}^{k+1}\left(s,\cdot\right)\right\Vert ^{2}\right)ds\\
	& \le C_{1}\log\left(\gamma\left(\beta\right)\right)\bar{x}\left\Vert v^{k}\right\Vert _{C\left(\left[0,\bar{x}\right];H^{1}\left(0,1\right)\right)}^{2}+C_{1}\log\left(\gamma\left(\beta\right)\right)\int_{0}^{x}\left\Vert v_{x}^{k+1}\left(s,\cdot\right)\right\Vert ^{2}ds.
	\end{align*}
	Therefore, using the Gronwall inequality we obtain
	\begin{align*}
	& \sup_{x\in\left[0,\bar{x}\right]}\left(\left\Vert v_{x}^{k+1}\left(x,\cdot\right)\right\Vert ^{2}+\left\Vert v_{y}^{k+1}\left(x,\cdot\right)\right\Vert ^{2}\right)\\
	& \le C_{1}\log\left(\gamma\left(\beta\right)\right)\bar{x}\left(\left\Vert v^{k}\right\Vert _{C\left(\left[0,\bar{x}\right];H^{1}\left(0,1\right)\right)}^{2}+\left\Vert v_{x}^{k}\right\Vert _{C\left(\left[0,\bar{x}\right];L^{2}\left(0,1\right)\right)}^{2}\right)\gamma^{C_{1}\bar{x}}\left(\beta\right).
	\end{align*}
	Choosing now $\bar{x}$ small enough such that
	\begin{align}\label{eta}
		\eta_{\beta}^2:=2C_{1}\log\left(\gamma\left(\beta\right)\right)\bar{x}\gamma^{C_{1}\bar{x}}\left(\beta\right)<1,
	\end{align}
	we then find that
	\[
	\left\Vert v^{k+1}\right\Vert _{C\left(\left[0,\bar{x}\right];H^{1}\left(0,1\right)\right)}^{2}+\left\Vert v_{x}^{k+1}\right\Vert _{C\left(\left[0,\bar{x}\right];L^{2}\left(0,1\right)\right)}^{2}\le\eta_{\beta}^{2}\left(\left\Vert v^{k}\right\Vert _{C\left(\left[0,\bar{x}\right];H^{1}\left(0,1\right)\right)}^{2}+\left\Vert v_{x}^{k}\right\Vert _{C\left(\left[0,\bar{x}\right];L^{2}\left(0,1\right)\right)}^{2}\right).
	\]
	Henceforth, for $r\ge 1$ it holds that
	\begin{align}
	& \left\Vert u_{\beta}^{\varepsilon,k+r}-u_{\beta}^{\varepsilon,k}\right\Vert _{C\left(\left[0,\bar{x}\right];H^{1}\left(0,1\right)\right)}+\left\Vert \partial_{x}u_{\beta}^{\varepsilon,k+r}-\partial_{x}u_{\beta}^{\varepsilon,k}\right\Vert _{C\left(\left[0,\bar{x}\right];L^{2}\left(0,1\right)\right)} \nonumber\\
	& \le\sum_{j=1}^{r}\left(\left\Vert u_{\beta}^{\varepsilon,k+j}-u_{\beta}^{\varepsilon,k+j-1}\right\Vert _{C\left(\left[0,\bar{x}\right];H^{1}\left(0,1\right)\right)}+\left\Vert \partial_{x}u_{\beta}^{\varepsilon,k+j}-\partial_{x}u_{\beta}^{\varepsilon,k+j-1}\right\Vert _{C\left(\left[0,\bar{x}\right];L^{2}\left(0,1\right)\right)}\right) \nonumber\\
	& \le\sum_{j=1}^{r}\eta_{\beta}^{k+j-1}\left(\left\Vert u_{\beta}^{\varepsilon,1}-u_{\beta}^{\varepsilon,0}\right\Vert _{C\left(\left[0,\bar{x}\right];H^{1}\left(0,1\right)\right)}+\left\Vert \partial_{x}u_{\beta}^{\varepsilon,1}-\partial_{x}u_{\beta}^{\varepsilon,0}\right\Vert _{C\left(\left[0,\bar{x}\right];L^{2}\left(0,1\right)\right)}\right) \nonumber\\
	& \le\frac{\eta_{\beta}^{k}\left(1-\eta_{\beta}^{r}\right)}{1-\eta_{\beta}}\left(\left\Vert u_{\beta}^{\varepsilon,1}\right\Vert _{C\left(\left[0,\bar{x}\right];H^{1}\left(0,1\right)\right)}+\left\Vert \partial_{x}u_{\beta}^{\varepsilon,1}\right\Vert _{C\left(\left[0,\bar{x}\right];L^{2}\left(0,1\right)\right)}\right).\label{cauchy}
	\end{align}
	This shows that $\left\{ u_{\beta}^{\varepsilon,k}\right\} _{k\in\mathbb{N}}$ is a Cauchy sequence in the space $\mathbb{H}_{\bar{x}}$. Thus, there exists uniquely $u_{\beta}^{\varepsilon}\in \mathbb{H}_{\bar{x}}$ such that $u_{\beta}^{\varepsilon,k}\to u_{\beta}^{\varepsilon}$ strongly in $\mathbb{H}_{\bar{x}}$ as $k\to \infty$. When $r\to\infty$ in (\ref{cauchy}), we have
	\[
	\left\Vert u_{\beta}^{\varepsilon,k}-u_{\beta}^{\varepsilon}\right\Vert _{\mathbb{H}_{\bar{x}}}\le\frac{\eta_{\beta}^{k}}{1-\eta_{\beta}}\left\Vert u_{\beta}^{\varepsilon,1}\right\Vert _{\mathbb{H}_{\bar{x}}}.
	\]
	Combining this convergence with the linearity of the stabilized operator, we arrive at
	\[
	\left\Vert \mathbf{P}_{\varepsilon}^{\beta}u_{\beta}^{\varepsilon,k}-\mathbf{P}_{\varepsilon}^{\beta}u_{\beta}^{\varepsilon}\right\Vert \le C_{1}\log\left(\gamma\left(\beta\right)\right)\left\Vert u_{\beta}^{\varepsilon,k}-u_{\beta}^{\varepsilon}\right\Vert _{H^{1}\left(0,1\right)}\le\frac{\eta_{\beta}^{k}C_{1}\log\left(\gamma\left(\beta\right)\right)}{1-\eta_{\beta}}\left\Vert u_{\beta}^{\varepsilon,1}\right\Vert _{\mathbb{H}_{\bar{x}}}.
	\]
	Thanks to the choice (\ref{eta}) and to the fact that (\ref{u1}) holds, $\mathbf{P}_{\varepsilon}^{\beta}u_{\beta}^{\varepsilon,k}\to\mathbf{P}_{\varepsilon}^{\beta}u_{\beta}^{\varepsilon}$ strongly in $L^2(0,1)$ as $k\to\infty$. Hence, the limit function $u_{\beta}^{\varepsilon}\in \mathbb{H}_{\bar{x}}$ obtained above is actually the solution of the regularized system (\ref{IVP1})--(\ref{IVP2}). We complete the proof of the theorem.
\end{proof}

It is worth mentioning that we can choose $\bar{x}=\frac{1}{2}\gamma^{-2C_1}(\beta)$ to fulfill the choice (\ref{eta}) because
\[
2C_{1}\log\left(\gamma\left(\beta\right)\right)\gamma^{C_{1}\bar{x}}\bar{x}=\log\left(\gamma^{C_{1}}\left(\beta\right)\right)\gamma^{-C_{1}}\left(\beta\right)\gamma^{C_{1}\left(\bar{x}-1\right)}<1.
\]
Now, we state the strong convergence of $u_{\beta}^{\varepsilon,k}$ towards the true solution $u$ by a combination of Theorems \ref{thm:err2}--\ref{thm:err3}. 

\begin{theorem}\label{thm:err5}
	Under the assumptions of Theorem \ref{thm:err2}, if the choice of $\gamma(\beta)$ in (\ref{gamma}) holds, then
	\[
	\left\Vert \partial_{x}u_{\beta}^{\varepsilon,k}\left(x,\cdot\right)-\partial_{x}u\left(x,\cdot\right)\right\Vert ^{2}+\left\Vert u_{\beta}^{\varepsilon,k}\left(x,\cdot\right)-u\left(x,\cdot\right)\right\Vert _{H^{1}\left(0,1\right)}^{2}\le C\left[\gamma^{\frac{7C_{1}}{3}-2}\left(\beta\right)+\eta_{\beta}^{2k}\right].
	\]
	Also, if the choice (\ref{gamma1}) holds, then
	\[
	\left\Vert u_{\beta}^{\varepsilon,k}\left(x,\cdot\right)-u\left(x,\cdot\right)\right\Vert ^{2}\le C\left[\gamma^{\frac{7C_{1}}{3}-2}\left(\beta\right)+\eta_{\beta}^{2k}\right].
	\]
\end{theorem}

\subsection{Generalizations}\label{subsec:5.3}
\subsubsection{Non-homogeneous cases}
Even though the inverse problem (\ref{original1}) has a simple form, in this part we show that solving more general equations is simply the same. Cf. \cite{Qian2008}, the Cauchy problem for a non-homogeneous elliptic equation can be reduced to the simplified one (\ref{original1}). Indeed, consider the following boundary value determination inverse problem:
\begin{align}\label{original}
\begin{cases}
w_{xx}+w_{yy}=f\left(x,y\right) & \text{in }\left(0,1\right)\times\left(0,1\right),\\
w\left(x,0\right)=g_0(x), w\left(x,1\right)= g_1(x) & \text{for }x\in [0,1],\\
w\left(0,y\right)=w_{0}\left(y\right),w_{x}\left(0,y\right)=w_{1}\left(y\right) & \text{for }y\in [0,1].
\end{cases}
\end{align}
In this regard, we can look for $w = u + \tilde{u}$, where $u$ satisfies the inverse problem (\ref{original1}) and $\tilde{u}$ obeys the following boundary value problem:
\begin{align}\label{utilde}
\begin{cases}
\tilde{u}_{xx}+\tilde{u}_{yy}=f\left(x,y\right) & \text{in }\left(0,1\right)\times\left(0,1\right),\\
\tilde{u}\left(x,0\right)=g_{0}\left(x\right),\tilde{u}\left(x,1\right)=g_{1}\left(x\right) & \text{for }x\in\left[0,1\right],\\
\tilde{u}_{x}\left(0,y\right)=w_{1}\left(y\right) & \text{for }y\in\left[0,1\right],\\
\tilde{u}\left(1,y\right)=\left(1-y\right)g_{0}\left(1\right)+yg_{1}\left(1\right) & \text{for }y\in\left[0,1\right].
\end{cases}
\end{align}
It is clear that (\ref{utilde}) is a two-dimensional elliptic equation with mixed non-homogeneous boundary conditions. Therefore, it is obviously well-posed with respect to all inputs. Henceforth, instead of solving the inverse problem (\ref{original}), we can simply investigate the simpler case (\ref{original1}), which facilitates a lot of computational issues in the inversion method. Multidimensional non-homogeneous problems can also be transformed into the Cauchy problem for the Laplace equation using the same transformation as above; cf. e.g. \cite{Elden2009}.

\subsubsection{Convergence with large noise ($\varepsilon \ge1$)}
This is now the first time we attempt to show convergence of the QR scheme with large noise. Here, we are interested in answering the question whether or not our QR scheme is convergent when $\varepsilon \gg 1$. To do so, we only need to modify our establishment for the perturbing and stabilized operators along with a new auxiliary function. Instead of using $\gamma(\beta)\ge 1$ as we have done with the case $\varepsilon\to 0$, we now consider an auxiliary function $\tau:(0,1)\to\mathbb{R}$ such that for $\beta\in(0,1)$ there holds
\[
\tau(\beta)\le 1,\quad \lim_{\beta\to 0}\tau(\beta)=0.
\]
Thereby, we plug $\gamma(\beta)=\frac{1}{\tau(\beta)}$ into Definitions \ref{def1}--\ref{def2} and proceed the mathematical analysis as in the case $\varepsilon\to 0$. In short, we state the principle convergence analysis in the following theorem, while the convergence regarding the choice of $\tau$ in terms of large noise $\varepsilon$ follows immediately.
\begin{theorem}\label{thm:err4}
	Suppose the measurement assumption (\ref{measure}) holds with $\varepsilon \gg 1$. If we choose $\tau(\beta)\le 1$ such that $2C_1\log\left(\frac{1}{\tau\left(\beta\right)}\right)/3>1$, then the following error estimate holds:
	\begin{align}\label{err4}
	& \left\Vert \partial_{x}u_{\beta}^{\varepsilon}\left(x,\cdot\right)-\partial_{x}u\left(x,\cdot\right)\right\Vert ^{2}+\left\Vert u_{\beta}^{\varepsilon}\left(x,\cdot\right)-u\left(x,\cdot\right)\right\Vert_{H^1(0,1)}^{2}\\
	& \le\frac{4}{3}C_{1}^{2}\left[4\varepsilon^{2}+\frac{27}{8}C_{0}^{2}C_{1}^{-3}\tau^{2}\left(\beta\right)\log^{-3}\left(\frac{1}{\tau(\beta)}\right)x\left\Vert u\right\Vert _{C\left([0,1];\mathbb{W}\right)}^{2}\right]\log^{2}\left(\frac{1}{\tau(\beta)}\right)\tau^{-7C_{1}x/3}(\beta).\nonumber
	\end{align}
\end{theorem}

Observe in (\ref{err4}) that when $\varepsilon \gg 1$ we use the small quantity $\tau^2\left(\beta\right)$ to control $\varepsilon^2$ as well as the term $\log^{2}\left(\frac{1}{\tau(\beta)}\right)$. Thereupon, the choices of $\tau(\beta)$ can be taken as in (\ref{gamma}) and (\ref{gamma1}), respectively. Finally, for intermediate noise like $\varepsilon\to 1$, we can still get the convergence by taking a shifted auxiliary function $\tau := \tau - 1$, for instance.


\section{Numerical results}\label{sec:6}
In this part, we take into account a finite difference solution of our proposed QR scheme in the linearized version we have studied above. Since the discretization for the domain of $x$ has to be dependent of $\varepsilon$ in the sense of (\ref{eta}), its mesh is understood as ``fine mesh" in this part. Meanwhile, we use ``coarse mesh" for the domain of $y$. For each $\varepsilon > 0$, we consider a uniform grid of mesh-points $x_{m}= m\Delta x$, where $\mathbb{N}\ni m\le M$ and $\Delta x$ is the equivalent mesh-width
in $x$. In the same manner, we take $y_{n}=n\Delta y$ with $\mathbb{N}\ni n\le N$. As mentioned in (\ref{Pnum}), the stabilized operator is chosen as
\[
\mathbf{P}_{\varepsilon}^{\beta}u=-2\sum_{\mu_{j}\le\frac{1}{16}\log^2\left(\gamma\left(\beta\right)\right)}\mu_{j}\left\langle u,\phi_{j}\right\rangle \phi_{j},
\]
which gives $C_1 = 1/2$.
Since we are interested in the choice (\ref{gamma1}), we follow the  arguments in Remark \ref{rem:4.3} to choose $\gamma(\beta) = \beta^{-1}$ and $\beta = \varepsilon$. By this way, we condition that $\Delta x = \frac{1}{2}\gamma^{-2C_1}(\beta) = \varepsilon/2$ and for simplicity, we take $\Delta x = \Delta y$.

In this linearization regime, we only need $k=3$ because, cf. Theorem \ref{thm:err5}, the convergence is eventually dominated by the H\"older rate in $\gamma$. For example, take $\varepsilon = 10^{-2}$ we find in (\ref{eta}) that
\[
\eta_{\beta}^2 = \frac{1}{2}\log\left(\frac{1}{\varepsilon}\right)\frac{\varepsilon}{\varepsilon^{\varepsilon/4}} \approx 0.0233.
\]
Therefore, we compute that $\eta_{\beta}^{2k}\approx 1.2\times 10^{-5}$, while the H\"older rate gives $\gamma^{\frac{7C_1}{3}-2}(\beta)=\varepsilon^{5/6}\approx 0.0215$.

As to the Dirichlet eigen-elements we have mentioned in Remark \ref{rem:2.3}, it is trivial to get that
\[
\phi_{j}\left(y\right) = \sqrt{2}\sin\left(j\pi y\right), \quad \mu_{j} = j^2\pi^2\quad \text{for }j\in \mathbb{N}.
\]
Prior to the derivation of the discrete version of (\ref{linear}), we note that we are not concerned with highly oscillatory integrals usually met in infinite series (cf. e.g. \cite{Khoa2017a}) because of the very low Fourier domain after truncation. Indeed, for $\varepsilon = 10^{-2}$ one has
\[
\mu_{j}\le\frac{1}{16}\log^2\left(\frac{1}{\varepsilon}\right)\approx 1.33, 
\]
which means $j\le 0.37$, i.e. $j=0$ in this case. It also holds up to $\varepsilon = 10^{-5}$. Note that when $j=0$, the stabilized operator vanishes. Below, we only consider $\varepsilon=10^{-1},10^{-2}$ (relatively meaning $10\%, 1\%$) as usually met in reality. In this regard, the number of mesh-points in $x$ and $y$ is not large and less expensive for real-world measurements.

For ease of presentation, we neglect the presence of $\beta$. Taking the starting point $u_{m,n}^{\varepsilon,0}=0$, we
seek a discrete solution $u_{m,n}^{\varepsilon,k}\approx u^{\varepsilon,k}\left(x_m,y_n\right)$ satisfying the following equation:
\begin{align}\label{discrete1}
\frac{u_{m+1,n}^{\varepsilon,k}-2u_{m,n}^{\varepsilon,k}+u_{m-1,n}^{\varepsilon,k}}{\left(\Delta x\right)^{2}}=\frac{u_{m,n+1}^{\varepsilon,k}-2u_{m,n}^{\varepsilon,k}+u_{m,n-1}^{\varepsilon,k}}{\left(\Delta y\right)^{2}}+2\Delta y\sum_{j\le\frac{1}{4\pi}\log\left(\frac{1}{\varepsilon}\right)}\sum_{l=0}^{N}\mu_{j}u_{m,l}^{\varepsilon,k-1}\phi_{j}\left(y_{l}\right)\phi_{j}\left(y_{n}\right).
\end{align}
Letting $r= \frac{\Delta x}{\Delta y}$, (\ref{discrete1}) is equivalent to
\begin{align*}
u_{m+1,n}^{\varepsilon,k}&=2u_{m,n}^{\varepsilon,k}-u_{m-1,n}^{\varepsilon,k}+r^{2}\left(u_{m,n+1}^{\varepsilon,k}-2u_{m,n}^{\varepsilon,k}+u_{m,n-1}^{\varepsilon,k}\right)\\ &
+2\left(\Delta x\right)^2\Delta y\sum_{j\le\frac{1}{4\pi}\log\left(\frac{1}{\varepsilon}\right)}\sum_{l=0}^{N}\mu_{j}u_{m,l}^{\varepsilon,k-1}\phi_{j}\left(y_{l}\right)\phi_{j}\left(y_{n}\right)\;\text{for }1\le m\le M-1,1\le n\le N -1.
\end{align*}
Henceforth, the fully discrete version of (\ref{linear})  can be written in matrix form as 
\begin{align}\label{discrete2}
\mathbf{U}_{m+1}^{\varepsilon,k}=\mathbf{K}\mathbf{U}_{m}^{\varepsilon,k}-\mathbf{U}_{m-1}^{\varepsilon,k}+\mathbf{F}\left(\mathbf{U}_{m}^{\varepsilon,k-1}\right)\quad\text{for }1\le m\le M-1,
\end{align}
where we have denoted by $\mathbf{U}_{m}^{\varepsilon,k}=\left(u_{m,1}^{\varepsilon,k},u_{m,2}^{\varepsilon,k},\ldots,u_{m,N-1}^{\varepsilon,k}\right)^{\text{T}}\in\mathbb{R}^{N-1}$ and $\mathbf{K}\in\mathbb{M}^{\left(N-1\right)\times\left(N-1\right)}$, $ \mathbf{F}\in \mathbb{R}^{N-1}$ given by
\begin{align*}
\mathbf{K}=\begin{bmatrix}2-2r^{2} & r^{2} & 0 & \cdots & 0\\
r^{2} & 2-2r^{2} & r^{2} & \cdots & 0\\
0 & \ddots & \ddots & \ddots & \vdots\\
\vdots &  & r^{2} & 2-2r^{2} & r^{2}\\
0 & \cdots & 0 & r^{2} & 2-2r^{2}
\end{bmatrix},\;\mathbf{F}\left(\mathbf{U}_{m}^{\varepsilon,k-1}\right)=\begin{bmatrix}\mathbf{F}\left(\mathbf{U}_{m}^{\varepsilon,k-1}\right)\left(y_{1}\right)\\
\mathbf{F}\left(\mathbf{U}_{m}^{\varepsilon,k-1}\right)\left(y_{2}\right)\\
\vdots\\
\vdots\\
\mathbf{F}\left(\mathbf{U}_{m}^{\varepsilon,k-1}\right)\left(y_{N-1}\right)
\end{bmatrix}.
\end{align*}
Here, the function $\mathbf{F}\left(\mathbf{U}_{m}^{\varepsilon,k-1}\right)$ can be computed in the previous step of linearization procedure, whose elements read as
\begin{align*}
\mathbf{F}\left(\mathbf{U}_{m}^{\varepsilon,k-1}\right)\left(y_{n}\right) & =2\left(\Delta x\right)^2\Delta y\sum_{j\le\frac{1}{4\pi}\log\left(\frac{1}{\varepsilon}\right)}\sum_{l=0}^{N}\mu_{j}u_{m,l}^{\varepsilon,k-1}\phi_{j}\left(y_{l}\right)\phi_{j}\left(y_{n}\right)\\
& =2\left(\Delta x\right)^2\Delta y\sum_{j\le\frac{1}{4\pi}\log\left(\frac{1}{\varepsilon}\right)}\mu_{j}\begin{bmatrix}\phi_{j}\left(y_{1}\right)\\
\phi_{j}\left(y_{2}\right)\\
\vdots\\
\vdots\\
\phi_{j}\left(y_{N-1}\right)
\end{bmatrix}^{\text{T}}\begin{bmatrix}u_{m,1}^{\varepsilon,k-1}\\
u_{m,2}^{\varepsilon,k-1}\\
\vdots\\
\vdots\\
u_{m,N-1}^{\varepsilon,k-1}
\end{bmatrix}\phi_{j}\left(y_{n}\right).
\end{align*}
Note that due to the Dirichlet condition in $y$, we have $u_{m,0}^{\varepsilon,k}=u_{m,N}^{\varepsilon,k}=0$ for $0\le m \le M$. Moreover, using the initial conditions we endow (\ref{discrete2}) with
\begin{align}\label{discrete3}
\mathbf{U}_{0}^{\varepsilon,k}=\mathbf{U}_{1}^{\varepsilon,k}=\left(u_{0}^{\varepsilon}\left(y_{1}\right),u_{0}^{\varepsilon}\left(y_{2}\right),\ldots,u_{0}^{\varepsilon}\left(y_{N-1}\right)\right)^{\text{T}},
\end{align}
which is attained by our measured data.

It is well-known that the system (\ref{discrete2})--(\ref{discrete3}) satisfies the von Neumann stability condition when $r\le 1$ due to the explicit finite difference regime we choose. This choice is based on the fact that  $\Delta x$ is already very small by its dependence on $\varepsilon$ and that the explicit scheme is easier to use in implementation. As to the choice of numerical integration involved in $\mathbf{P}_{\varepsilon}^{\beta}$ we simply rely on the Riemann sum because we intend to take $\Delta y $ slightly small for better images' resolution. The number of $y_n$ is not a matter here since we can apply some quadrature methods to get good accuracy for, e.g., $N\le 10$. This is already postulated in our previous work \cite{Khoa2017a} and is not our main interest in this work.

Our numerical illustrations consist of two (2) tests, where we suppose to know the true solutions with different shapes. We remark that these true solutions satisfy the non-homogeneous elliptic equation, which turns out that we solve the general inverse problem (\ref{original}) instead of just (\ref{original1}). The algorithm that transforms (\ref{original}) into (\ref{original1}) is already given in subsection \ref{subsec:5.3}. With the analytical solutions, we can easily compute all inputs involved in (\ref{original}). On the other hand, we take into account the following type of additive measured data of (\ref{original1}):
\begin{align}
&\label{noise1} u_{0}^{\varepsilon}\left(y\right)=u\left(0,y\right)+\text{rand}\left(y\right)\varepsilon=w_{0}\left(y\right)-\tilde{u}\left(0,y\right)+\text{rand}\left(y\right)\varepsilon,\\
& \partial_{y}u_{0}^{\varepsilon}\left(y\right)=\partial_{y}w_{0}\left(y\right)-\partial_{y}\tilde{u}\left(0,y\right)+\text{rand}\left(y\right)\varepsilon,\label{noise2}
\end{align}
where $w_0(y)$ is known from the analytical solution, $\tilde{u}\left(0,y\right)$ can be found by solving (\ref{utilde}), and $\text{rand}$ is a uniformly distributed random number such that
$\max_{y\in\left[0,1\right]}\text{rand}\left(y\right)\le 1$. By this way, we fulfill the assumption (\ref{measure}).

Lastly, we define the $\ell^2$ and relative errors as
\begin{align*}
& E_{\ell^{2}}=\sqrt{\frac{1}{(M+1)(N+1)}\sum_{m=0}^{M}\sum_{n=0}^{N}\left|w_{\beta}^{\varepsilon}\left(x_{m},y_{n}\right)-w_{\text{true}}\left(x_{m},y_{n}\right)\right|^{2}},\\
& E_{\text{rel}}=\frac{\sqrt{\sum_{m=0}^{M}\sum_{n=0}^{N}\left|w_{\beta}^{\varepsilon}\left(x_{m},y_{n}\right)-w_{\text{true}}\left(x_{m},y_{n}\right)\right|^{2}}}{\sqrt{\sum_{m=0}^{M}\sum_{n=0}^{N}\left|w_{\text{true}}\left(x_{m},y_{n}\right)\right|^{2}}}\times100\%.
\end{align*}

\begin{remark}
	In practice, we just need to use the measured data (\ref{noise1}) to estimate its gradient (\ref{noise2}) because measurements are usually expensive. At the discretization level, one has
	\begin{align*}
	\partial_{y}u_{0}^{\varepsilon}\left(y_{n}\right) & \approx\frac{u_{0}^{\varepsilon}\left(y_{n+1}\right)-u_{0}^{\varepsilon}\left(y_{n}\right)}{\Delta y}=\frac{u_{0}\left(y_{n+1}\right)+\text{rand}\left(y_{n+1}\right)\varepsilon-u_{0}\left(y_{n}\right)-\text{rand}\left(y_{n}\right)\varepsilon}{\Delta y}\\
	& \approx\partial_{y}u_{0}\left(y_{n}\right)+\varepsilon\left(\Delta y\right)^{-1}\left(\text{rand}\left(y_{n+1}\right)-\text{rand}\left(y_{n}\right)\right).
	\end{align*}
	If we take $\left(\Delta y\right)^{-1} = C_1 \log\left(\gamma(\beta)\right)/3=\log\left(\varepsilon^{-1}\right)/3$
	, we have
	\[
	\left|\varepsilon\left(\Delta y\right)^{-1}\left(\text{rand}\left(y_{n+1}\right)-\text{rand}\left(y_{n}\right)\right)\right|\le\varepsilon\rho_{\beta},
	\]
	where we have recalled $\rho_{\beta}$ in the proof of Theorem \ref{thm:err1}. Theoretically, the bound $\varepsilon\rho_{\beta}$ is acceptable because cf. (\ref{intialest}), the whole error bound for initial data in the proof remains unchanged. Note that by this $\varepsilon$ dependence, the number of measurement points in $y$ is small. Therefore, as we have mentioned above, one should apply, e.g., the Gauss--Legendre method to get a fine numerical integration for $\mathbf{P}_{\varepsilon}^{\beta}$. This reveals a facing challenge of inverse problems in real-world applications.
\end{remark}

\subsection{Test 1: sinusoidal humps}
In this test, we reconstruct the heat distribution satisfying the inverse problem (\ref{original}), where we suppose that the true analytical solution is
\[
w_{\text{true}}\left(x,y\right)=\sin\left(6x\right)\sin\left(6y\right).
\]
This test models sinusoidal humps, which is one of the classical examples in simulation. To validate the proposed scheme, we depict in Figure \ref{fig:1} the computed solutions with $\varepsilon = 10\%$ and $\varepsilon = 1\%$ and also, we compare them with the true solution illustrated therein. We can see that the computed solutions are very close to the true one. Besides, we briefly report that for $\varepsilon = 10\%$ the $\ell^2$ error is $0.09$ and the relative error is $15.2\%$. For $\varepsilon = 1\%$, the errors reduce to $0.06$ and $10.8\%$, respectively.
\begin{figure}
	\centering
	\subfloat[Computed ($\varepsilon=10^{-1}$)]{\includegraphics[scale=0.5]{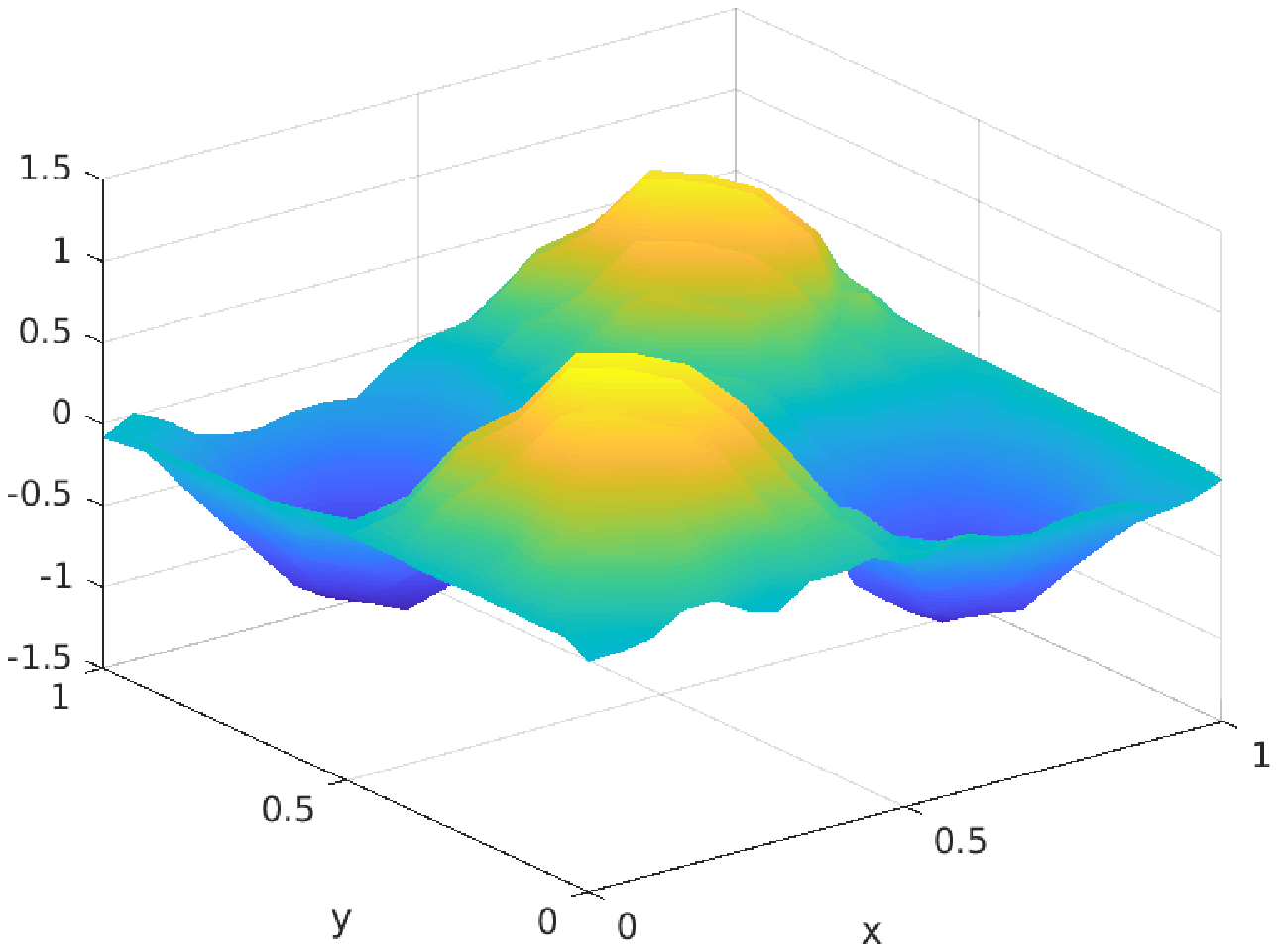}	
	}
	\subfloat[Computed ($\varepsilon=10^{-1}$)]{\includegraphics[scale=0.5]{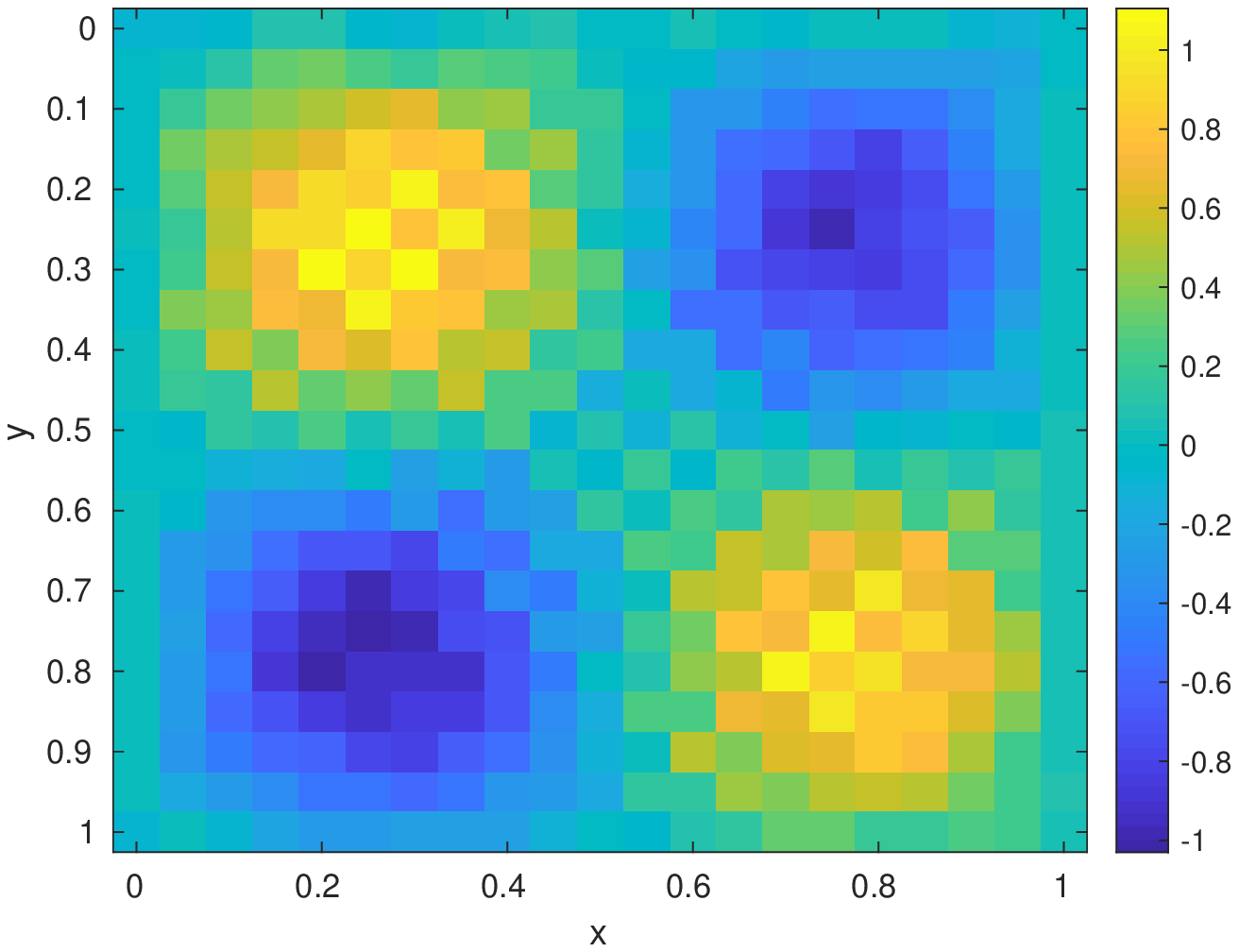}	
	}\\
	\subfloat[Computed ($\varepsilon=10^{-2}$)]{\includegraphics[scale=0.5]{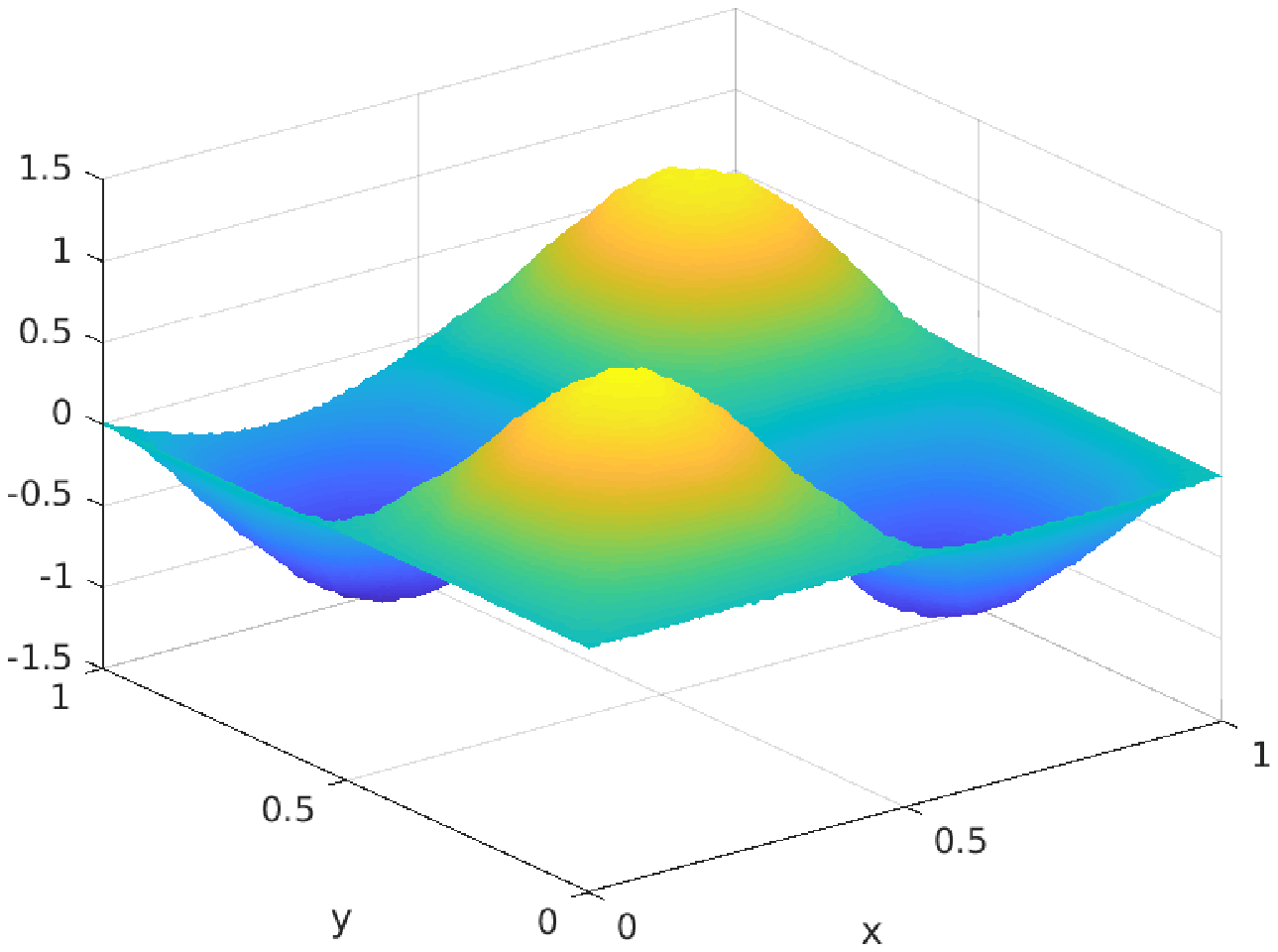}	
	}
	\subfloat[Computed ($\varepsilon=10^{-2}$)]{\includegraphics[scale=0.5]{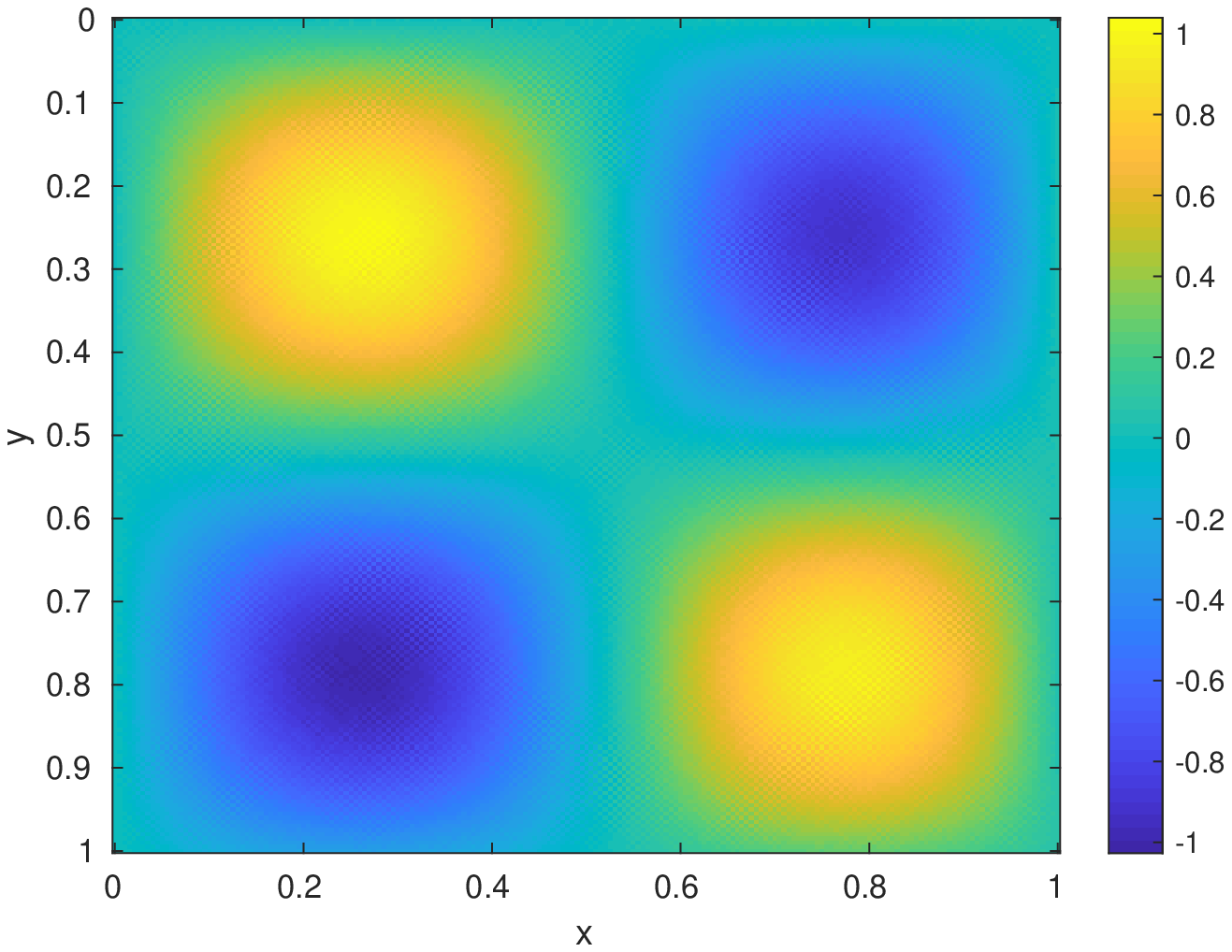}	
	}\\
	\subfloat[True]{\includegraphics[scale=0.5]{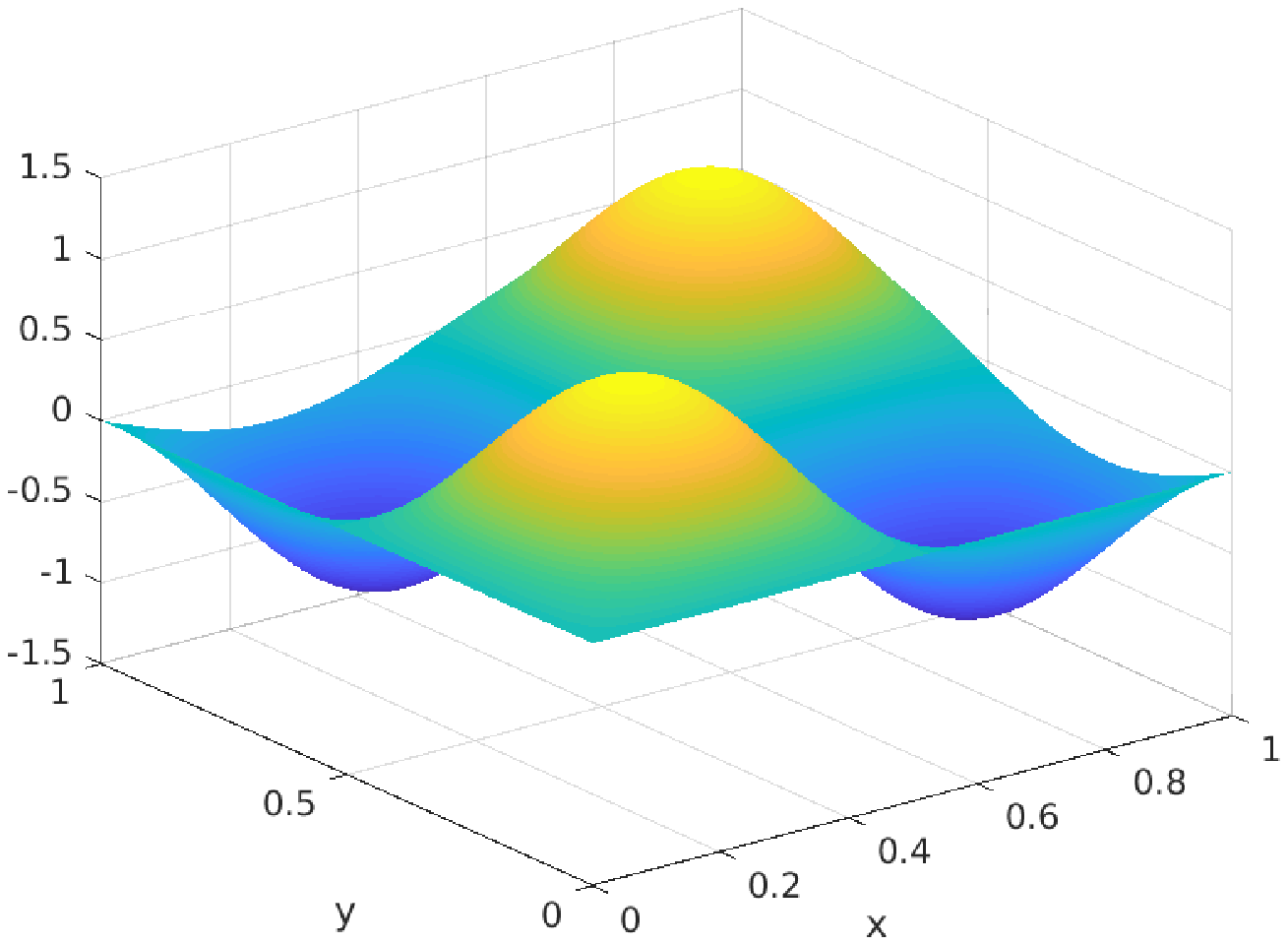}	
	}
	\subfloat[True]{\includegraphics[scale=0.5]{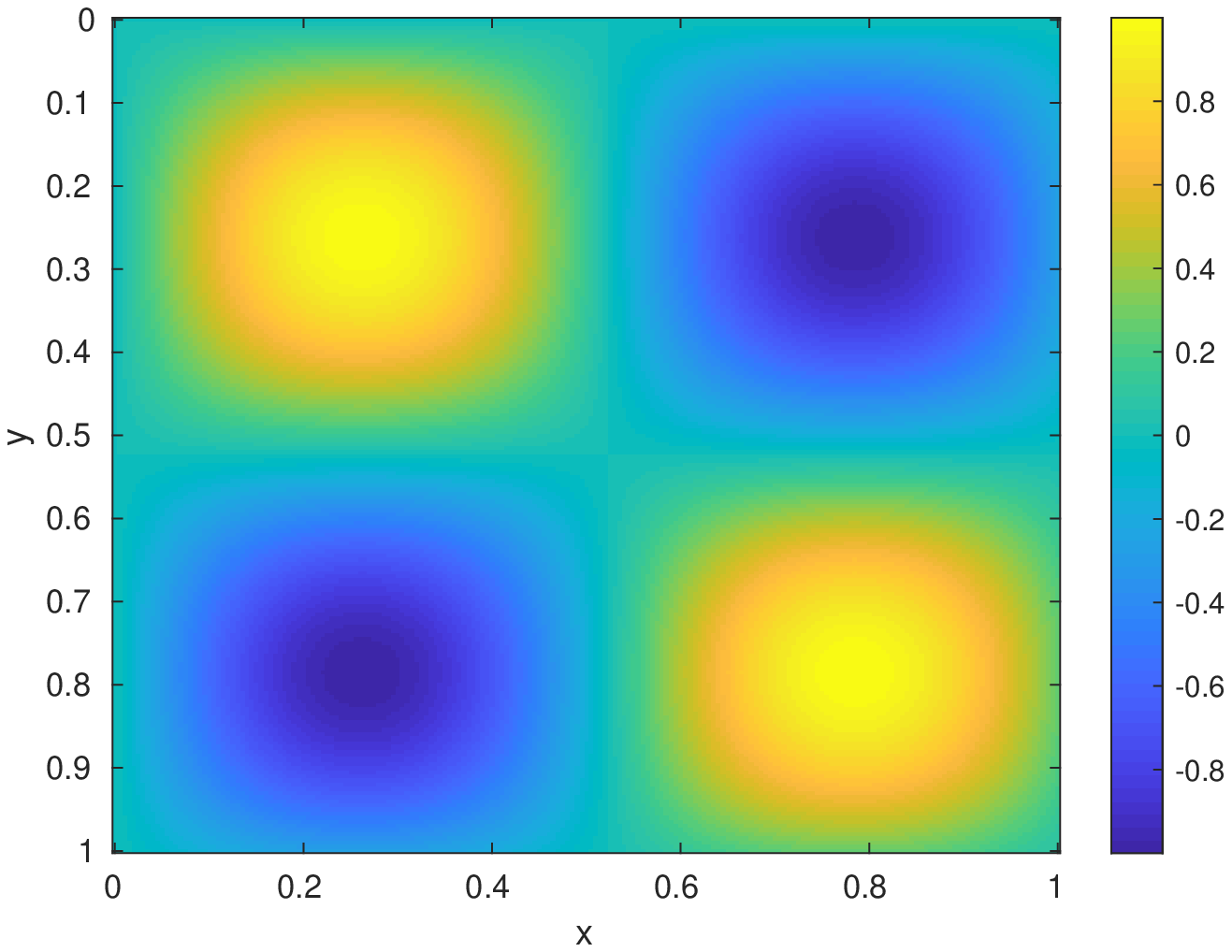}	
	}
	\caption{(a)--(b) Reconstructed solution in Test 1 with $\varepsilon=10^{-1},M=N=20$. (c)--(d) Reconstructed solution in Test 1 with $\varepsilon=10^{-2},M=N=200$. (e)--(f) Illustrations of the true solution in Test 1.\label{fig:1}}
\end{figure}
\subsection{Test 2: a box-shaped protrusion}
In the second test, we suppose that
\[
w_{\text{true}}\left(x,y\right)=\frac{1}{0.001+\left(x-0.5\right)^{4}+\left(y-0.5\right)^{4}},
\]
which resembles a scaled ``witch of Agnesi". Unlike Test 1, the true solution in this test attains a really huge maximal value. Therefore, the $\ell^2$ error can be relatively large. Figure \ref{fig:2} shows illustrations of the computed and true solutions. We observe that for a very coarse mesh in both $x$ and $y$ ($M=N=20$) and large noise (10\%) the 3D image of the computed solution is very accurate. We remark that the projection of the true solution onto the plane $\left\{z=0\right\}$ is box-shaped. Then in this particular comparison, the accuracy of the approximate solution is also enjoyed when $\varepsilon = 1\%$. Furthermore, as to the numerical errors we find that for $\varepsilon = 10\%$ the $\ell^2$ and relative errors are   $5.55$ and $1.9\%$, respectively. When $\varepsilon = 1\%$, they become smaller with $2.45$ and merely $0.81\%$.

\begin{figure}
	\centering
	\subfloat[Computed ($\varepsilon=10^{-1}$)]{\includegraphics[scale=0.5]{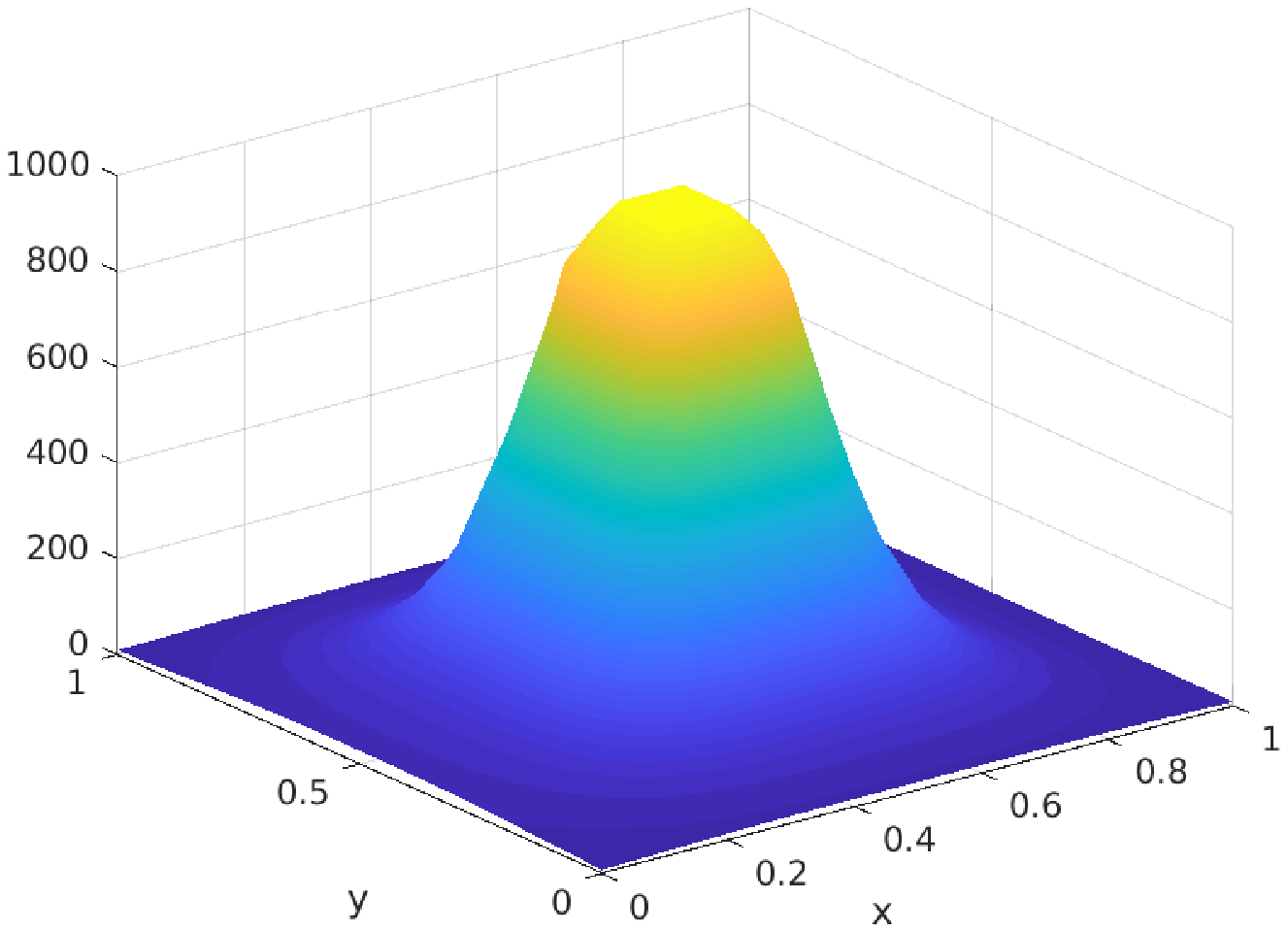}	
	}
	\subfloat[Computed ($\varepsilon=10^{-1}$)]{\includegraphics[scale=0.5]{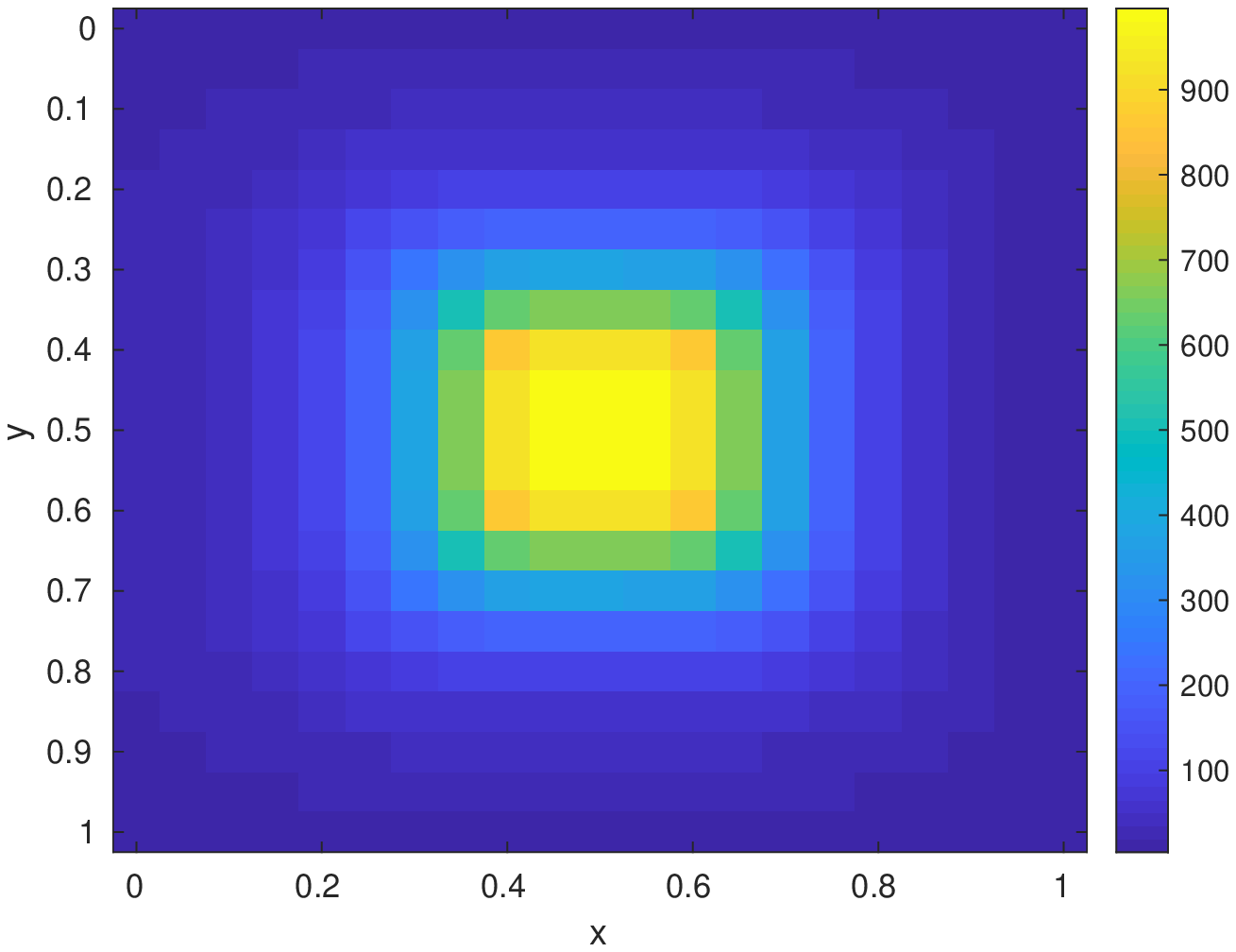}	
	}\\
	\subfloat[Computed ($\varepsilon=10^{-2}$)]{\includegraphics[scale=0.5]{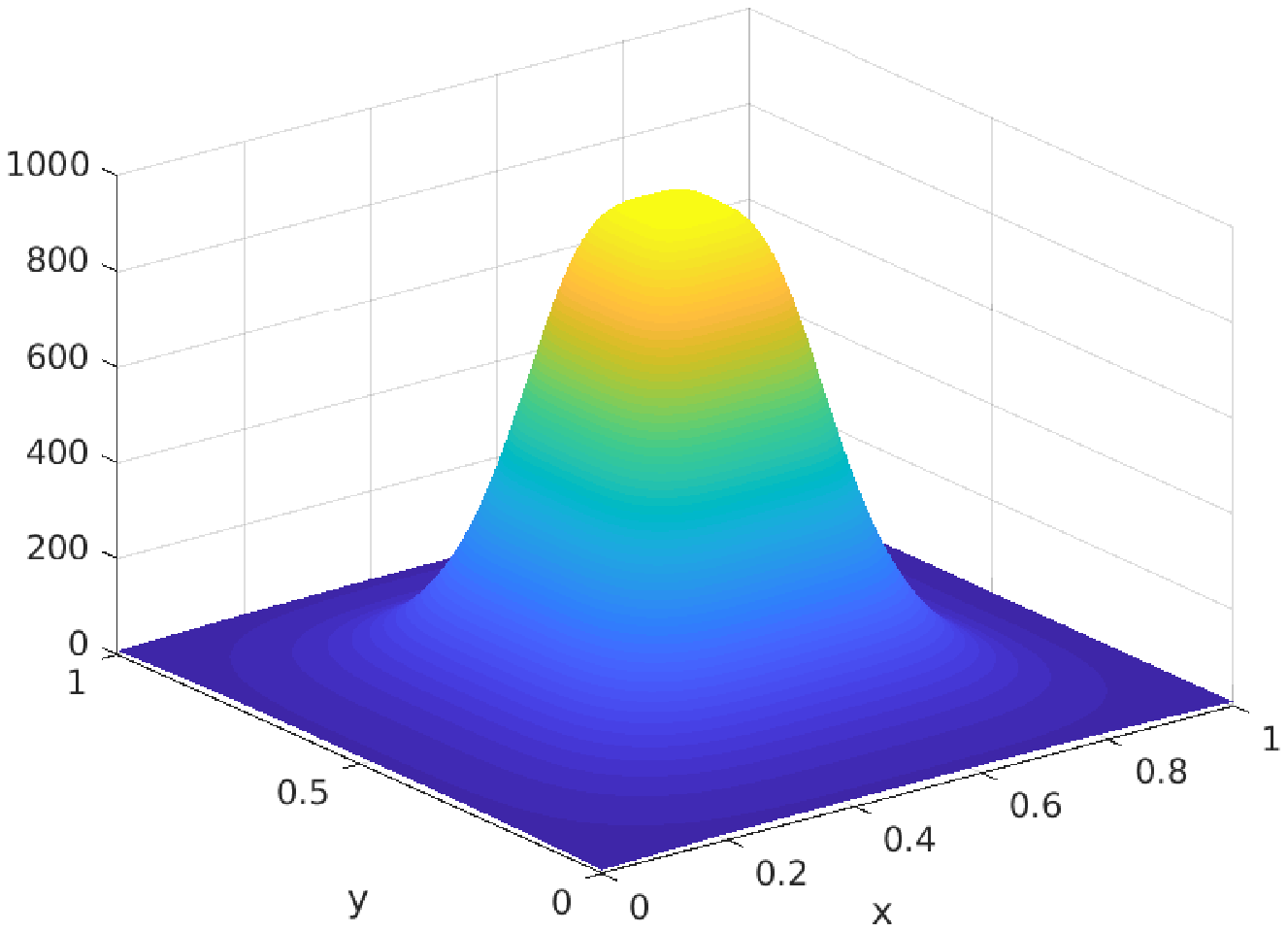}	
	}
	\subfloat[Computed ($\varepsilon=10^{-2}$)]{\includegraphics[scale=0.5]{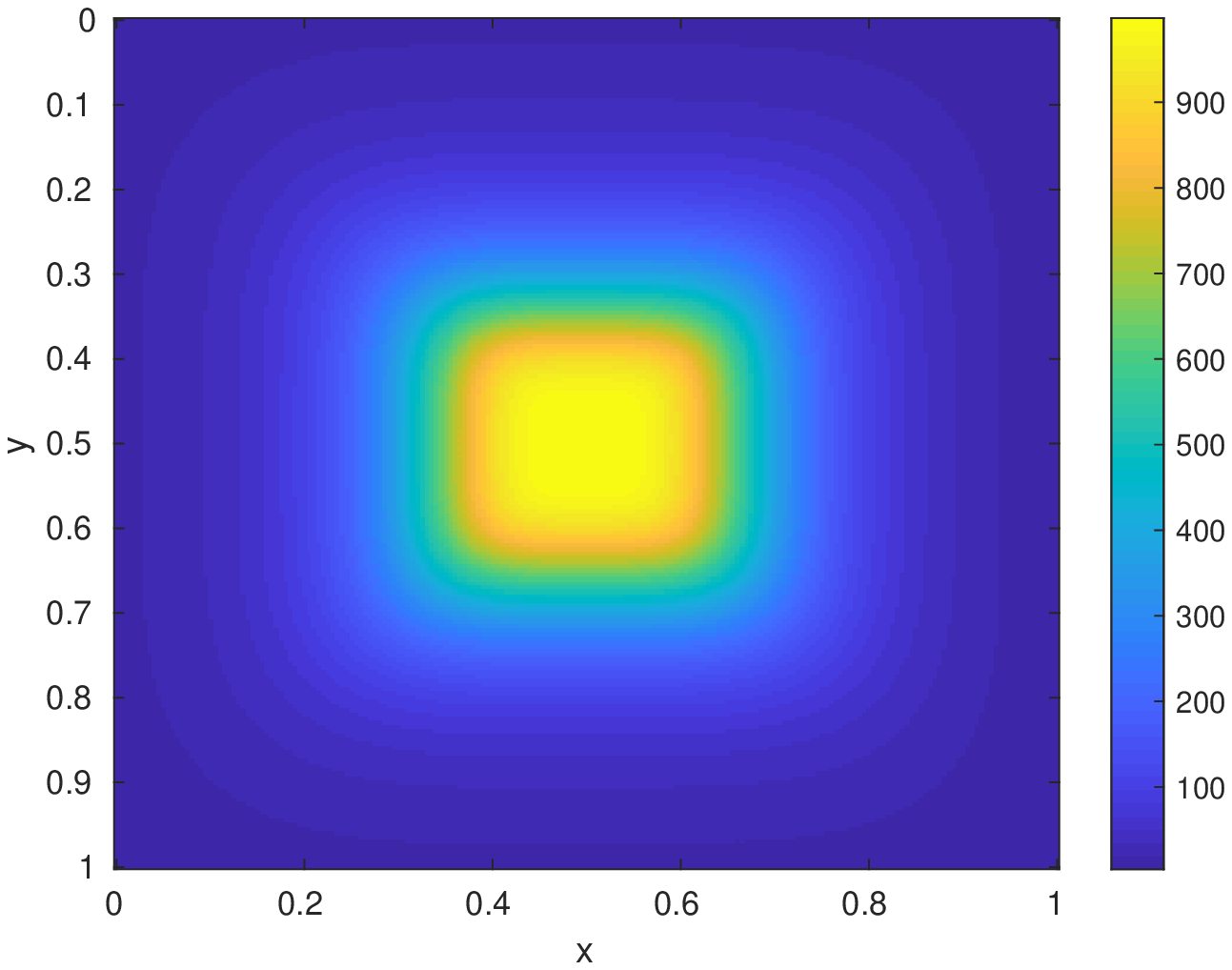}	
	}\\
	\subfloat[True]{\includegraphics[scale=0.5]{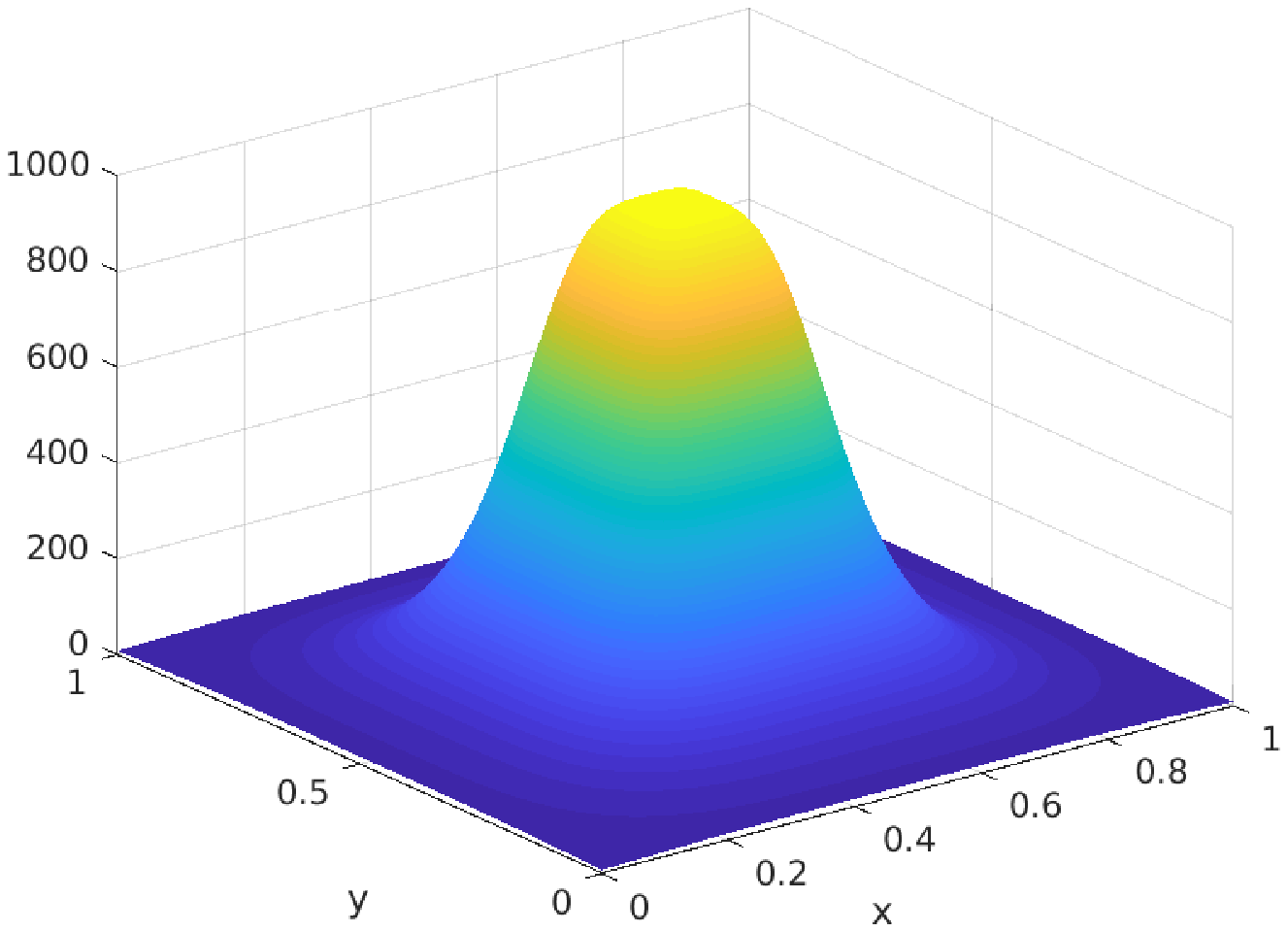}	
	}
	\subfloat[True]{\includegraphics[scale=0.5]{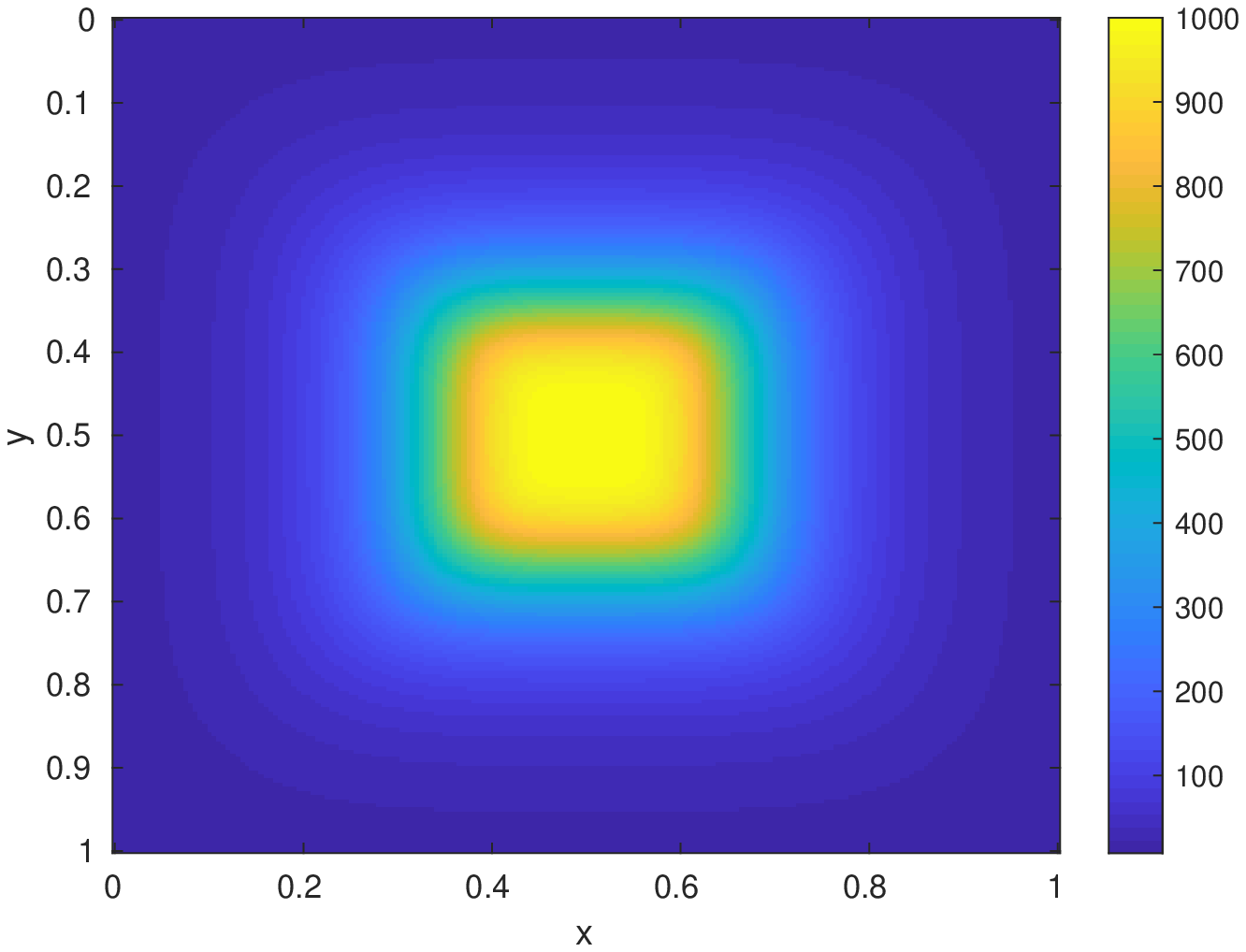}	
	}
	\caption{(a)--(b) Reconstructed solution in Test 2 with $\varepsilon=10^{-1},M=N=20$. (c)--(d) Reconstructed solution in Test 2 with $\varepsilon=10^{-2},M=N=200$. (e)--(f) Illustrations of the true solution in Test 2.\label{fig:2}}
\end{figure}

\section{Conclusions}\label{sec:7}
We have demonstrated that the modified quasi-reversibility method we have developed so far for regularization of time-reversed parabolic problems is well-adapted to solve the Cauchy problem for elliptic equations. In this regard, we have proposed two conditional estimates for our perturbing and stabilized operators, which are crucial for the strong convergence of the regularization scheme for this problem. The notion behind this method is to turn the inverse problem into the forward-like problem, which can be solved by several numerical methods. In this case, we solve the Cauchy problem for the elliptic equation by a linear wave equation. We have exploited the energy method to prove weak solvability of the regularized problem. Moreover, driven by a Carleman-type function, we have obtained the H\"older rate of convergence towards the ``ideal" true solution with distinctive cases of noise levels: small ($\varepsilon \ll 1$), large ($\varepsilon \gg 1$) and intermediate. This shows the flexibility of our method.

Some upcoming topics should be investigated in the near future.
\begin{enumerate}
	\item The present choice of the perturbing and stabilized operators is really dependent of information of the eigen-elements. However, solving the Sturm--Liouville eigenvalue problems posed in complex-geometry materials (cf. e.g. \cite{Grebenkov2013,Liu2013}) is rather challenging, compared to standard media like rectangle and circle we usually consider in theory. Therefore, our first forthcoming target is about a numerical solution of the regularized system in more complicated domains.
	
	\item To show that our regularization results have general applicability in applied sciences, we are inclined to adapt this new method to cope with many other types of inverse problems. For instance, it is promising to study this method for coefficient inverse problems \cite{Klibanov2019} and inverse fractional PDEs \cite{Trong2019,Siskova2019}.
	
	\item If we take into account the strong solution of the regularized problem, our QR scheme is convergent in $H^2$. As in \cite{Klibanov1991,Bourgeois2006}, it is then interesting to deduce an error estimation in $H^2$.
\end{enumerate}

\appendix

\section{Proofs of (\ref{rela}) and (\ref{rela1})}
\label{appendix-sec1}
To prove (\ref{rela}) and (\ref{rela1}), we recall the formulation of the Fourier coefficients of the true solution, which usually indicates the natural ill-posedness of the Cauchy problem for elliptic equations. Cf. \cite{Khoa2017}, we compute for the Laplace system (\ref{original1}) that
\begin{align*}
 \left\langle u\left(x,\cdot\right),\phi_{j}\right\rangle =\cosh\left(\sqrt{\mu_{j}}x\right)\left\langle u_{0},\phi_{j}\right\rangle ,\quad
 \left\langle u_{x}\left(x,\cdot\right),\phi_{j}\right\rangle =\sqrt{\mu_{j}}\sinh\left(\sqrt{\mu_{j}}x\right)\left\langle u_{0},\phi_{j}\right\rangle,
\end{align*}
which yield
\begin{align*}
& \left\langle u\left(x,\cdot\right),\phi_{j}\right\rangle +\frac{\left\langle u_{x}\left(x,\cdot\right),\phi_{j}\right\rangle }{\sqrt{\mu_{j}}}=e^{\sqrt{\mu_{j}}x}\left\langle u_{0},\phi_{j}\right\rangle ,\\
& \left\langle u\left(1,\cdot\right),\phi_{j}\right\rangle +\frac{\left\langle u_{x}\left(1,\cdot\right),\phi_{j}\right\rangle }{\sqrt{\mu_{j}}}=e^{\sqrt{\mu_{j}}}\left\langle u_{0},\phi_{j}\right\rangle .
\end{align*}
Henceforth, we obtain the relation (\ref{rela}). Proof of (\ref{rela1}) follows because
\[
\frac{1}{\sqrt{\mu_{j}}}\left\langle u\left(x,\cdot\right),\phi_{j}\right\rangle \left\langle u_{x}\left(x,\cdot\right),\phi_{j}\right\rangle =\cosh\left(\sqrt{\mu_{j}}x\right)\sinh\left(\sqrt{\mu_{j}}x\right)\left|\left\langle u_{0},\phi_{j}\right\rangle \right|^{2}\ge0.
\]

\section*{Acknowledgment}
V.A.K was partly supported by the Research Foundation-Flanders (FWO) in Belgium under the project named ``Approximations for forward and inverse reaction-diffusion problems related to cancer models''. V.A.K is very grateful for the guidance and support of Prof. Dr. Loc Hoang Nguyen (Charlotte, USA) during the fellowship at UNCC.



\bibliographystyle{elsarticle-num}

\bibliography{sample}

\end{document}